\documentclass[10pt]{amsart}
\usepackage{amsfonts}
\usepackage{mathrsfs}
\usepackage{amscd}
\usepackage{amsmath}
\usepackage{amssymb}
\usepackage{appendix}
\usepackage{latexsym}
\usepackage{lscape}
\usepackage{xypic}

\newtheorem{thm}{Theorem}[section]
\newtheorem{prop}[thm]{Proposition}
\newtheorem{lem}[thm]{Lemma}

\newtheorem{lem-def}[thm]{Lemma-Definition}
\newtheorem{thm-def}[thm]{Theoem-Definition}
\newtheorem{cor}[thm]{Corollary}
\newtheorem{rmk}[thm]{Remark}
\theoremstyle{remark}
\newtheorem{ex}[thm]{Example}

\newtheorem{Rmk}{Remark}[section]
\theoremstyle{definition}
\newtheorem{dfn}{Definition}[section]

\numberwithin{equation}{section}

\input xy
\xyoption{all}

\newcommand{\quash}[1]{}  
\newcommand{\nc}{\newcommand}
\nc{\on}{\operatorname}

\newcommand{\fraka}{{\mathfrak a}}

\newcommand{\frakh}{{\mathfrak h}}

\newcommand{\frakl}{{\mathfrak l}}

\newcommand{\fraks}{{\mathfrak s}}
\newcommand{\frakt}{{\mathfrak t}}

\newcommand{\frakz}{{\mathfrak z}}

\newcommand{\bbA}{{\mathbb A}}
\newcommand{\bbB}{{\mathbb B}}
\newcommand{\bbC}{{\mathbb C}}

\newcommand{\bbF}{{\mathbb F}}
\newcommand{\bbG}{{\mathbb G}}

\newcommand{\bbN}{{\mathbb N}}

\newcommand{\bbP}{{\mathbb P}}
\newcommand{\bbQ}{{\mathbb Q}}
\newcommand{\bbR}{{\mathbb R}}

\newcommand{\bbZ}{{\mathbb Z}}

\newcommand{\calB}{{\mathcal B}}
\newcommand{\calC}{{\mathcal C}}
\newcommand{\calD}{{\mathcal D}}
\newcommand{\calE}{{\mathcal E}}
\newcommand{\calF}{{\mathcal F}}
\newcommand{\calG}{{\mathcal G}}
\newcommand{\calH}{{\mathcal H}}

\newcommand{\calL}{{\mathcal L}}

\newcommand{\calO}{{\mathcal O}}
\newcommand{\calP}{{\mathcal P}}

\newcommand{\calR}{{\mathcal R}}
\newcommand{\calS}{{\mathcal S}}

\newcommand{\calX}{{\mathcal X}}
\newcommand{\calY}{{\mathcal Y}}
\newcommand{\calZ}{{\mathcal Z}}

\nc{\al}{{\alpha}} \nc{\be}{{\beta}}
\newcommand{\ga}{{\gamma}}
\nc{\ve}{{\varepsilon}} \nc{\Ga}{{\Gamma}}
\newcommand{\la}{{\lambda}}
\nc{\La}{{\Lambda}}

\nc{\ad}{{\on{ad}}}
\newcommand{\Ad}{{\on{Ad}}}
\nc{\Adm}{{\on{Adm}}} \nc{\aff}{{\on{aff}}}
\nc{\Aff}{{\mathbf{Aff}}}
\newcommand{\Aut}{{\on{Aut}}}
\nc{\Bun}{{\on{Bun}}}

\nc{\der}{{\on{der}}}

\nc{\diag}{{\on{diag}}}
\newcommand{\End}{{\on{End}}}
\nc{\Fl}{{\calF\ell}}
\newcommand{\Gal}{{\on{Gal}}}
\newcommand{\Gr}{{\on{Gr}}}
\nc{\rH}{{\on{H}}}
\newcommand{\Hom}{{\on{Hom}}}
\nc{\IC}{{\on{IC}}}
\newcommand{\id}{{\on{id}}}
\nc{\Id}{{\on{Id}}}
\newcommand{\ind}{{\on{ind}}}
\nc{\Ind}{{\on{Ind}}}
\newcommand{\Lie}{{\on{Lie}}}

\nc{\Rep}{{\on{Rep}}}
\newcommand{\Res}{{\on{Res}}}
\nc{\res}{{\on{res}}} \nc{\Sat}{{\on{Sat}}}
\newcommand{\s}{{\on{sc}}}

\newcommand{\Spec}{{\on{Spec}}}
\nc{\Sph}{\on{Sph}}

\nc{\tr}{{\on{tr}}}
\newcommand{\Tr}{{\on{Tr}}}
\newcommand{\Mod}{{\mathrm{-Mod}}}
\newcommand{\GL}{{\on{GL}}}
\nc{\GSp}{{\on{GSp}}}

\nc{\SL}{{\on{SL}}} \nc{\SU}{{\on{SU}}} \nc{\SO}{{\on{SO}}}

\newcommand{\bGr}{{\overline{\Gr}}}
\nc{\bFl}{{\overline{\Fl}}} \nc{\bU}{{\overline{U}}}
\nc{\wGr}{{\widetilde{\Gr}}}

\nc{\ppars}{(\!(s)\!)}
\newcommand{\ppart}{(\!(t)\!)}
\newcommand{\pparu}{(\!(u)\!)}

\def\xcoch{\mathbb{X}_\bullet}
\def\xch{\mathbb{X}^\bullet}

\title{The geometric Satake correspondence for ramified groups}
\author{Xinwen Zhu}

\address{Timo Richarz: Mathematisches Institut der Universit\"at Bonn, Endenicher Allee 60, 53115 Bonn, Germany}
\email{richarz@math.uni-bonn.de}
\address{Xinwen Zhu: Department of Mathematics, Northwestern University, 2033 Sheridan Road, Evanston, IL 60208, USA}
\email{xinwenz@math.northwestern.edu}


\begin{document}
\date{July 2011}
\maketitle \maketitle

\begin{abstract}
We prove the geometric Satake isomorphism for a reductive group
defined over $F=k\ppart$, and split over a tamely ramified
extension. As an application, we give a description of the nearby
cycles on certain Shimura varieties via the Rapoport-Zink-Pappas
local models.
\end{abstract}

\tableofcontents

\section*{Introduction}
The Satake isomorphism (for unramified groups) is the starting point
of the Langlands duality. Let us first recall its statement. Let $F$
be a non-archimedean local field with ring of integers $\calO$ and
residue field $k$, and let $G$ be a connected unramified reductive
group over $F$ (e.g. $G=\GL_n$). Let $A\subset G$ a maximal split
torus of $G$, and $W_0$ be the Weyl group of $(G,A)$. Let $K$ be a
hyperspecial subgroup of $G(F)$ containing $A(\calO)$ (e.g.
$K=\GL_n(\calO)$). Then the classical Satake isomorphism describes
the spherical Hecke algebra $\Sph=C_c(K\!\setminus\! G(F)/K)$, the
algebra of compactly supported bi-$K$-invariant functions on $G(F)$
under convolution. Namely, there is an isomorphism of algebras
\[\Sph\simeq \bbC[\xcoch(A)]^{W_0},\]
where $\xcoch(A)$ is the coweight lattice of $A$, and
$\bbC[\xcoch(A)]^{W_0}$ denotes the $W_0$-invariants of the group
algebra of $\xcoch(A)$.

If $F$ has positive characteristic $p>0$, then the classical Satake
correspondence has a vast enhancement. For simplicity, let us assume
that $G$ is split over $F$ (for the general case, see Theorem
\ref{unramified Satake}). Let us write $G=H\otimes_k F$ for some
split group $H$ over $k$ so that $K=H(\calO)$. Let
$\Gr_H=H(F)/H(\calO)$ be the affine Grassmannian of $H$. Choose
$\ell$ a prime different from $p$, and let $\Sat_H$ be the category
of $(K\otimes \bar{k})$-equivariant perverse sheaves with
$\overline{\bbQ}_\ell$-coefficients on $\Gr_H\otimes\bar{k}$. Then
this is a Tannakian category and there is an equivalence
\[\Sat_H\simeq\Rep (G^\vee_{\overline{\bbQ}_\ell}),\]
where $G^\vee_{\overline{\bbQ}_\ell}$ is the dual group of $G$ and
$\Rep (G^\vee_{\overline{\bbQ}_\ell})$ is the tensor category of
algebraic representations of $G^\vee_{\overline{\bbQ}_\ell}$ (cf.
\cite{Gi,MV}).

There is also a version of Satake isomorphism for an arbitrary
reductive group over $F$, as recently proved by Haines and Rostami
(cf. \cite{HR})\footnote{There is another version, known earlier, as
in \cite{Car}.}. Namely, let $\calB(G)$ be the Bruhat-Tits building
of $G$ and $v\in \calB(G)$ be a special vertex. Let $K_v\subset
G(F)$ be the special parahoric subgroup of $G(F)$ corresponding to
$v$. Let $A$ be a maximal split $F$-torus of $G$ such that
$K_v\supset A(\calO)$, let $M$ be the centralizer of $A$ in $G$ and
$W_0=N_G(A)/M$ be the Weyl group as before. Let $M_1$ be the unique
parahoric subgroup of $M(F)$, and $\La_M=M(F)/M_1$, which is a
finitely generated abelian group. Then
\begin{equation}\label{ramified}
C_c(K_v\!\setminus\! G(F)/K_v)\simeq \bbC[\La_M]^{W_0}.
\end{equation}
More explicitly, suppose that $G$ is quasi-split so that $M=T$ is a
maximal torus. Then \[\La_M=(\xcoch(T)_I)^\sigma,\] where $I$ is the
inertial group and $\sigma$ is the Frobenius, and
$(\xcoch(T)_I)^\sigma$ denotes the $\sigma$-invariants of the
$I$-coinvariants of the group $\xcoch(T)$.

The goal of this paper is to provide a geometric version of the
above isomorphism when $F$ has positive characteristic $p$ and the
group $G$ is quasi-split and splits over a \emph{tamely} ramified
extension. More precisely, let $k$ be an algebraically closed field
and let $\ell\neq \on{char}k$ be a prime. Let $G$ be a group over
the local field $F=k\ppart$ (so that $G$ is quasi-split
automatically), which is split over a tamely ramified extension.
That is, there is a finite extension $\tilde{F}/F$ such that
$G_{\tilde{F}}$ is split and $\on{char}k\nmid[\tilde{F}:F]$. Let
$v\in\calB(G)$ be a special vertex in the building of $G$ and let
$\underline{G}_{v}$ be the parahoric group scheme over
$\calO=k[[t]]$ (in the sense of Bruhat-Tits), determined by $v$. We
write $LG$ for the loop space of $G$ and $K_v=L^+\underline{G}_v$
for the jet space of $\underline{G}_v$. By definition, for any
$k$-algebra $R$, $LG(R)=G(R\hat{\otimes}_k F)$ and
$K_v(R)=\underline{G}_v(R\hat{\otimes}_k\calO)$. Let
\[{\Fl_v}=LG/K_v\]
be the (twisted) affine flag variety\footnote{One would call $\Fl_v$
the affine Grassmannian of $G$. However, we reserve the name
``affine Grassmannian" of $G$ for another object, as defined in
Definition \ref{affGrass}.}, which is an ind-scheme over $k$. Let
$\calP_v=\calP_{K_v}({\Fl_v})$ be the category of $K_v$-equivariant
perverse sheaf on ${\Fl_v}$, with coefficients in
$\overline{\bbQ}_\ell$. Let $H$ be a split Chevalley group over
$\bbZ$ such that $G\otimes_FF^s\simeq H\otimes F^s$, where $F^s$ is
a (fixed) separable closure of $F$. Then there is a natural action
of $I=\Gal(F^{s}/F)$ on $H^\vee:=H^\vee_{\overline{\bbQ}_\ell}$
(preserving a fixed pinning).

\begin{thm}\label{main}
The category $\calP_v$ has a natural tensor structure. In addition,
as tensor categories, there is an equivalence
\[\calR\calS: \Rep((H^\vee)^I)\simeq\calP_v,\]
such that $\rH^*\circ\calR\calS$ is isomorphic to the forgetful
functor, where $\rH^*$ is the hypercohomology functor.
\end{thm}

This theorem can be regarded as a categorification of
\eqref{ramified} in the case when $k$ is algebraically closed and
the group splits over a tamely ramified extension of $k((t))$.

Let us point out the following remarkable facts when the group is
ramified. First, the group $(H^\vee)^I$ is not necessarily connected
as is shown in Remark \eqref{non-connected}. Second, it is
well-known that if $G$ is unramified over $F$, then all the
hyperspecial subgroups of $G$ are conjugate under $G_{\ad}(F)$
(\cite[\S 2.5]{T}), where $G_\ad$ is the adjoint group of $G$.
However, this is no longer true for special parahoric of $G$ if $G$
is ramified. An example is given by the odd ramified unitary
similitude group $\on{GU}_{2m+1}$. There are essentially two types
of special parahorics of $\on{GU}_{2m+1}$, as given in
\eqref{spodd}. One of them has reductive quotient $\on{GO}_{2m+1}$
(denoted by $\underline{G}_{v_0})$, and the other has reductive
quotient $\on{GSp}_{2m}$ (denoted $\underline{G}_{v_1}$).
Accordingly, the geometry of the corresponding flag varieties
$\Fl_{v_0}$ and $\Fl_{v_1}$ are very different, while
$\calP_{v_0}\simeq\calP_{v_1}$. Indeed, their Schubert varieties
(i.e. closures of $K_{v_i}$-orbits) are both parameterized by
irreducible representations of $\on{GO}_{2m+1}$. Let
$\Fl_{v_0\bar{\mu}_{2m,1}}$ (resp. $\Fl_{v_1\bar{\mu}_{2m,1}}$) be
the Schubert variety in $\Fl_{v_0}$ (resp. $\Fl_{v_1}$)
parameterized by the standard representation of $\on{GO}_{2m+1}$.
Then it is shown in \cite{Z} that $\Fl_{v_0\bar{\mu}_{2m,1}}$ is not
Gorenstein, while in \cite{Ri} that $\Fl_{v_1\bar{\mu}_{2m,1}}$ is
smooth. On the other hand, the intersection cohomology of both
varieties gives the standard representation of $\on{GO}_{2m+1}$. In
addition, the stalk cohomologies of both sheaves are the ``same".
See Theorem \ref{BK} below.

\begin{Rmk}\label{max vs parahoric}
Instead of considering a special parahoric $K_v$ of $LG$, one can
begin with the special maximal``compact" $K'_v$, (i.e.
$K'_v=L^+\underline{G}'_v$, where $\underline{G}'_v$ is the
stabilizer group scheme of $v$ as constructed by Bruhat-Tits), and
consider the category of $K'_v$-equivariant perverse sheaves on
$LG/K'_v$. However, from geometric point of view, this is less
natural since $K'_v$ is not necessarily connected and the category
of $K'_v$-equivariant perverse sheaves is complicated. In fact, we
do not how to relate this category to the Langlands dual group yet.
In addition, when we discuss the Langlands parameters in Sect.
\ref{parameter}, it is also more``correct" to
consider $K_v$ rather than $K'_v$. 
\end{Rmk}

\medskip

The idea of the proof of the theorem is as follows. Using
Gaitsgory's nearby cycle functor construction as in \cite{G1,Z}, we
construct a functor
\[\calZ:\Sat_H\to\calP_v,\]
which is a central functor in the sense of \cite{B}. By standard
arguments in the theory of Tannakian equivalence and the
Mirkovic-Vilonen theorem, this already implies that
$\calP_v\simeq\Rep(\tilde{G}^\vee)$ for certain closed subgroup
$\tilde{G}^\vee\subset H^\vee$. Then we identify $\tilde{G}^\vee$
with $(H^\vee)^I$ using the parametrization of the $K_v$-orbits on
${\Fl_v}$.

\begin{Rmk}(i) We believe that the same argument (maybe with small
modifications) should work for groups split over wild ramified
extensions. However, we have not checked this carefully.

(ii) Our approach is more inspired by \cite{G1} rather than
\cite{MV}. However, it would be interesting to know whether there is
the similar theory of MV-cycles in the ramified case. It seems that
the geometry of semi-infinite orbits on $\Fl_v$ is similar to the
unramified case, except when $\Fl_v$ corresponds to one type of
special parahorics for odd unitary groups (the one denoted by
$\underline{G}_{v_1}$ as above). We do not know what happens in this
last case.
\end{Rmk}

When the group $G$ is quasi-split over the non-archimedean local
field $F=\bbF_q\ppart$ and $v$ is a special vertex of $\calB(G,F)$,
the affine flag variety $\Fl_v$ is defined over $\bbF_q$. We assume
that $v$ is very special, i.e. it remains special when we base
change $G$ to $\overline{\bbF}_q\ppart$ (see \S \ref{parameter} for
more discussions of this notion). Then we can consider the category
of $K_v$-equivariant semi-simple perverse sheaves on $\Fl_v$, pure
of weight zero, and denote it by $\calP_v^0$. On the other hand, let
$I$ be the inertial group of $F$ and $\sigma$ be the Frobenius of
$\Gal(k/\bbF_q)$, where $k=\overline{\bbF}_q$. Then the action of
$\Gal(F^s/F)$ on $H^\vee$ (via the pinned automorphisms) induces a
canonical action of $\sigma$ on $(H^\vee)^I$, denoted by
$\on{act}^{alg}$. One can form the semidirect product
$(H^\vee)^I\rtimes_{\on{act}^{alg}}\Gal(k/\bbF_q)$, which can be
regarded as a proalgebraic group over $\overline{\bbQ}_\ell$, and
consider the category of algebraic representations of
$(H^\vee)^I\rtimes_{\on{act}^{alg}}\Gal(k/\bbF_q)$, denoted by
$\Rep((H^\vee)^I\rtimes_{\on{act}^{alg}}\Gal(k/\bbF_q))$.

\begin{thm}\label{main finite}
In this case, the functor $\calR\calS$ in Theorem \ref{main} can be
extended to an equivalence
\[\calR\calS: \Rep((H^\vee)^I\rtimes_{\on{act}^{alg}}\Gal(k/\bbF_q))\simeq \calP_v^0,\]
whose composition with $\rH^*$ is isomorphic to the forgetful
functor.
\end{thm}

Let us mention that under this equivalence, the restriction to
$\Gal(k/\bbF_q)$ of the representation
$(H^\vee)^I\rtimes_{\on{act}^{alg}}\Gal(k/\bbF_q)$ on $\rH^*(\calF)$
for $\calF\in\calP^0_v$ is NOT the natural Galois action of
$\Gal(k/\bbF_q)$ on $\rH^*(\calF)$. However, their difference can be
described explicitly. See Sect. \ref{proof} and Appendix for more
details.

\medskip

Our next result is to use the ramified geometric Satake isomorphism
to obtain the stalk cohomology of sheaves on ${\Fl_v}$ (i.e. the
corresponding Lusztig-Kato polynomial in ramified case), following
an idea of Ginzburg (cf. \cite{Gi}). Let us state the result
precisely. The centralizer of $A$ in our case is a maximal torus of
$G$, denoted by $T$. Then the $K$-orbits on $\Fl_v$ are labeled by
$\xcoch(T)_I/W_0$, $W_0$-orbits of the coinvariants of the
cocharacter group of $T$. For $\bar{\mu}\in\xcoch(T)_I$, let
$\mathring{\Fl}_{v\bar{\mu}}$ be the corresponding orbit. For a
representation $V$ of $(H^\vee)^I$, let $V(\bar{\mu})$ be the weight
space of $V$ for $(T^\vee)^I$. Let $X^\vee\in\Lie (H^\vee)^I$ be a
certain principal nilpotent element (see Sect. \ref{stalk} for the
details), which induces a filtration $F_i V(\bar{\mu})=(\ker
X^\vee)^{i+1}\cap V(\bar{\mu})$ on $V(\bar{\mu})$, called the
Brylinski-Kostant filtration. Then we have

\begin{thm}\label{BK}
For $V\in \Rep((H^\vee)^I)$, let $\calR\calS(V)\in\calP_v$ be the
corresponding sheaf. Then
\[\dim \calH^{2i-(2\rho,\bar{\mu})}\calR\calS(V)|_{\Fl_{v\bar{\mu}}}=\dim \on{gr}^F_iV(\bar{\mu}).\]
\end{thm}
Here $\calH^*$ denotes the cohomology sheaves, and $2\rho$ is the
sum of positive roots of $H$, see Sect. \ref{reminder} for the
meaning of $(2\rho,\bar{\mu})$.

\medskip

One of our main motivations of this work is to apply these results
to the calculation of the nearby cycles of certain ramified unitary
Shimura varieties, via the Rapoport-Zink-Pappas local models. For
example, we obtain the following theorem (see Sect. \ref{Shimura}
for details).

\begin{thm}
Let $G=\on{GU}(r,s)$ be a unitary similitude group associated to an
imaginary quadratic extension $F/\bbQ$ and a hermitian space
$(W,\phi)$ over $F/\bbQ$. Let $p>2$ be a prime where $F/\bbQ$ is
ramified and the hermitian form is split. Let $K_p$ be a special
parahoric subgroup of $G=G(\bbQ_p)$. Let $K=K_pK^p\subset
G(\bbQ_p)G(\bbA_f^p)$ be a compact open subgroup with $K^p$ small
enough. Let $Sh_K$ be the associated Shimura variety over the reflex
field $E$ and $Sh_{K_p}$ be the integral model of $Sh_K$ over
$\calO_{E_p}$ (as defined in \cite{PR2}). Then for $\ell\neq p$, the
action of the inertial subgroup $I$ of $\Gal(\overline{\bbQ}_p/F_p)$
on the nearby cycle
$\Psi_{Sh_{K_p}\otimes_{\calO_{E_p}}\calO_{F_p}}(\bbQ_\ell)$ is
trivial.
\end{thm}

By applying Theorem \ref{BK}, it will not be hard to determine the
traces of Frobenius on these sheaves explicitly, which will be the
input of the Langlands-Kottwitz method to determine the local Zeta
factors of $Sh_K$. Instead, we characterize these traces of
Frobenius in terms of Langlands parameters, which verifies a
conjecture of Haines and Kottwitz in this case (see Proposition
\ref{Haines-Kottwitz}).

\begin{Rmk}(i)
While the definition of the integral model of a PEL-type Shimura
variety at an ``unramified" prime $p$ (i.e. the group is unramified
at $p$ and $K_p$ is hyperspecial) is well-known (cf. \cite{Ko1}),
the definition of such a model at the ramified prime $p$ (even for
$K_p$ special) is a subtle issue. In \cite{Pa,PR2}, the integral
models $Sh_{K_p}$ are defined as certain closed subschemes of
certain moduli problems of abelian varieties. Except a few cases
(e.g. $(r,s)=(n-1,1)$ and $n=r+s$ is small), there is no moduli
description of $Sh_{K_p}$ so far. In general, $Sh_{K_p}$ are not
smooth. Indeed, as shown in \cite{Pa,PR2}, when $n=r+s$ is odd and
$(r,s)=(n-1,1)$, for the special parahoric $K_p$ of $G(\bbQ_p)$ with
reductive quotient $\on{GO}_{n}$, $Sh_{K_p}$ is not even
semi-stable.

(ii) If $r\neq s$, then we know that $E=F$ and the above theorem
gives a complete description of the monodromy on the nearby cycles
of $Sh_{K_p}$. If $r=s$, then $E=\bbQ$, and the complete description
of the monodromy is more complicated. See Sect. \ref{Shimura} for
details. In any case, the action of inertia on the nearby cycle is
semi-simple.

(iii) We hope that there will be a ``good" compactification of such
Shimura varieties $Sh_{K_p}$. Then the above theorem, together with
the existence of such compactification, would imply that the
monodromy of $H_c^*(Sh_K\otimes_{E_p} F_p)$ is trivial.

(iv)The triviality of the monodromy as above would have the
following surprising consequence for the Albanese of Picard modular
surfaces. Namely, in the case when $(r,s)=(2,1)$, $F/\bbQ$ is
ramified at $p>2$ and $K_p=G(\bbQ_p)$ is a special parahoric, the
Albanese $\on{Alb}(Sh_{K_p})$ of $Sh_{K_p}$ is trivial. It will be
interesting to find the ``optimal" level structure at $p$ so that
$\on{Alb}(Sh_{K_p})$ can be possibly non-trivial. More detailed
discussion will appear elsewhere.
\end{Rmk}

\medskip

Let us quickly describe the organization of the paper. We will prove
Theorem \ref{main} and Theorem \ref{main finite} in \S
\ref{reminder}-\ref{proof}. Then we prove Theorem \ref{BK} in \S
\ref{stalk}.

In \S \ref{parameter}, we briefly discuss the Langlands parameters
associated to a smooth representation of a quasi-split $p$-adic
group, which has a vector fixed by a special parahoric. We call them
``spherical" representations, and we will see that their Langlands
parameters can be described easily. Again, the correct point of view
is to consider the special parahoric rather than the special maximal
compact. Then in \S \ref{Shimura}, we apply the previous results to
study the nearby cycles on certain unitary Shimura varieties.

The paper contains an appendix, joint with T. Richarz, where we
recover the full Langlands dual group via the Tannakian formalism.
In particular, we give the geometric Satake correspondence for
unramified groups. We hope that this formulation will be of
independent interest. In addition, we observe that for a reductive
group defined over $k$, the Tannakian formulism provides a natural
action $\on{act}^{geom}$ of $\Gal(k^s/k)$ on the dual group
$G^\vee$, which differs from the usual pinned action
$\on{act}^{alg}$ of $\Gal(k^s/k)$ on $G^\vee$ by the twist of ``the
half sum of the positive roots". This gives a geometric explanation
of the two natural normalizations of the Satake parameters.

\medskip

\noindent{\bf Notations.} Let $k$ be a field. We denote by $k^s$ a
separable closure of $k$.

For a (not necessarily connected) diagonalizable algebraic group $C$
defined over a field $F$, we denote by $\xch(C)$ its group of
characters and by $\xcoch(C)$ its group of cocharacters over the
separable closure $F^s$. The Galois group $I=\Gal(F^s/F)$ acts on
$\xch(C)$ (resp. $\xcoch(C)$) and the invariants (resp.
coinvariants) are denoted by $\xch(C)^I$ (resp. $\xch(C)_I,
\xcoch(C)^I, \xcoch(C)_I$). We will always use $\la,\mu,\ldots$ to
denote elements in $\xch(C)$ or $\xcoch(C)$ and
$\bar{\la},\bar{\mu}$ to denote elements in $\xch(C)_I$ or
$\xcoch(C)_I$. In general, let $I$ be a group acting on a set $S$.
We denote $S^I$ to be the subset of fixed points.

If $G$ is an algebraic group defined over a field $E$, we denote by
$\Rep(G)$ the category of finite dimensional representations of $G$
over $E$. If $G$ is connected reductive, we denote
$G_{\ad},G_{\der},G_{\s}$ to be its adjoint quotient, its derived
group, and the simply-connected cover of its derived group.

Let $k$ be a field and $\calO=k[[t]], F=k\ppart:=k[[t]][t^{-1}]$.
For an $\calO$-scheme $X$, we denote $L^+X$ to be the jet space over
$k$ so that for any $k$-algebra $R$, $L^+X(R)=X(R[[t]])$. For an
$F$-scheme $X$, we denote $LX$ to be its loop space so that
$LX(R)=X(R\ppart)$. If $X$ is defined over $k$, we write $L^+X$ for
$L^+(X\otimes\calO)$ and $LX$ for $L(X\otimes F)$ if no confusion
will arise.

For a variety $X$ over $k$, we denote $D(X)$ the usual (bounded)
derived category of $\ell$-adic sheaves on $X$
($\ell\nmid\on{char}k$). If $X=\lim X_i$ is an ind-scheme of
ind-finite type, $D(X)=\lim D(X_i)$ as usual. If there is an action
of an algebraic group $G$ on $X$, the $G$-equivariant derived
category is denoted by $D_G(X)$ (see \cite{BL} for the details). All
the functors like $f_*,f_!,f^*,f^!$ are understood in the derived
sense unless otherwise specified.

\medskip

\noindent{\bf Acknowledgement.} The author would like to thank D.
Gaitsgory, T. Haines, Y. Liu, I. Mirkovi\'{c}, G. Pappas, M.
Rapoport, T. Richarz, E. Urban, Z. Yun for useful discussions. The
author also thanks the hospitality of Tsinghua University, where
part of the work is done. The work of the author is supported by the
NSF grant under DMS-1001280.

\section{Reminders on the affine flag variety associated to a special
parahoric}\label{reminder} In this section, we collect basic facts
about the affine flag varieties associated to a special parahoric of
$G$. Another purpose of this section is to fix notations that are
used in the later sections.

Let $k$ be an algebraically closed field and let $G$ be a group over
the local field $F=k\ppart$, which is split over a tamely ramified
extension. Let us choose $A$ to be a maximal $F$-split torus of $G$
and $T$ be its centralizer. Then $T$ is a maximal torus of $G$ since
$G$ is quasi-split. Let us choose a rational Borel subgroup
$B\supset T$.

Let $H$ be a split Chevalley group over $\bbZ$ such that $H\otimes
F^s\simeq G\otimes F^s$. We need to choose this isomorphism
carefully. Let us fix a pinning $(H,B_H,T_H,X)$ of $H$ over $\bbZ$.
Let us recall that this means that $B_H$ is a Borel subgroup of $H$,
$T_H$ is a split maximal torus contained $B_H$, and
$X=\Sigma_{\tilde{a}\in\Delta_H} X_{\tilde{a}}\in\Lie B$, where
$\Delta_H$ is the set of simple roots of $H$,
$\tilde{U}_{\tilde{a}}$ is the root subgroup corresponding to
$\tilde{a}$ and $X_{\tilde{a}}$ is a generator in the rank one free
$\bbZ$-module $\Lie \tilde{U}_{\tilde{a}}$. Let $\Xi$ be the group
of pinned automorphisms of $(H,B_H,T_H,X)$, which is canonically
isomorphic to the group of the automorphisms of the root datum
$(\xch(T_H),\Delta_H,\xcoch(T_H),\Delta^\vee_H)$.

Let us choose an isomorphism $(G,B,T)\otimes_F\tilde{F}\simeq
(H,B_H,T_H)\otimes_{\bbZ}\tilde{F}$, where $\tilde{F}/F$ is a cyclic
extension such that $G\otimes\tilde{F}$ splits. This induces an
isomorphism of the root data
$(\xch(T_H),\Delta_H,\xcoch(T_H),\Delta^\vee_H)\simeq(\xch(T),\Delta,\xcoch(T),\Delta^\vee)$.
Now the action of $I=\Gal(\tilde{F}/F)$ on $G\otimes_F\tilde{F}$
induces a homomorphism $\psi:I\to\Xi$. Then we can always choose an
isomorphism
\begin{equation}\label{rho}
(G,B,T)\otimes_{F}\tilde{F}\simeq (H,B_H,T_H)\otimes_\bbZ \tilde{F}
\end{equation}
such that the action of $\ga\in I$ on the left hand side corresponds
to $\psi(\ga)\otimes\ga$. In the rest of the paper, we fix such an
isomorphism.

Recall that the Kottwitz homomorphism $\kappa:T(F)\to \xcoch(T)_I$
(cf. \cite{Ko,HRa}) induces an isomorphism
\[\xcoch(T)_I\simeq T(F)/\underline{T}^{\flat,0}(\calO),\]
where $\underline{T}^{\flat,0}$ is the unique parahoric group scheme
of $T$ over $\calO$ (the connected N\'{e}ron model). Our convention
of Kottwitz homomorphism is that the action of $t\in T(F)$ on
$A(G,A)$ (the apartment associated to $(G,A)$) is given by $v\mapsto
v-\kappa(t)$. Let $W_0=W(G,A)$ be the relative Weyl group of $G$. It
acts on $T$ and therefore on $\xcoch(T)_I$. In addition, its action
on the torsion subgroup $\xcoch(T)_{I,\on{tor}}\subset\xcoch(T)_I$
is trivial.

The Borel subgroup $B$ determines a set of positive roots
$\Phi^+=\Phi(G,A)^+$ for $G$. There is a natural map $\xcoch(T)_I\to
\xcoch(T)_I\otimes\bbR\simeq\xcoch(A)_\bbR$. We define the set of
dominant elements in $\xcoch(T)_I$ to be
\begin{equation}\label{plus}
\xcoch(T)_I^+=\{\bar{\mu}| (\bar{\mu},a)\geq 0 \mbox{ for }
a\in\Phi^+\}.
\end{equation}
Then the natural map
$\xcoch(T)_I^+\subset\xcoch(T)_I\to\xcoch(T)_I/W_0$ is bijective.
Let us define an order $\preceq$ on $\xcoch(T)_I$ as follows. Let
$Q_H$ be the coroot lattice for $H$. The action of $I$ on $Q_H$ will
send the positive coroots of $H$ (determined by the chosen Borel) to
positive coroots. Therefore, it makes sense to talk about positive
elements in $(Q_H)_I$. Namely, an element in $(Q_H)_I$ is positive
if its preimage in $Q_H$ is a sum of positive coroots (of $H$).
Since $(Q_H)_I\subset\xcoch(T)_I$, we can define for
$\la,\mu\in\xcoch(T)_I$,
\begin{equation}\label{order}
\bar{\la}\preceq\bar{\mu} \mbox{ if } \bar{\mu}-\bar{\la} \mbox{ is
positive in } (Q_H)_I.
\end{equation}

Let $\underline{G}_v$ be a special parahoric group scheme of $G$
over $\calO=k[[t]]$ in the sense of Bruhat-Tits (see \cite{T} for a
summary of the theory), such that the natural inclusion $A\subset G$
extends to $A_\calO\subset\underline{G}_v$ (i.e. the vertex $v$ in
the (reduced) building of $G$ corresponding to $\underline{G}_v$ is
contained in the apartment $A(G,A)$. For examples of such group
schemes, we refer to Sect. \ref{Shimura}. We write
$K_v=L^+\underline{G}_v$, and consider the (twisted) affine flag
variety
\[{\Fl_v}=LG/K_v.\]
This is an ind-projective scheme (cf. \cite{PR}). As is shown in
\emph{loc. cit.}, when $G$ is semi-simple and simply-connected,
$\Fl_v$ is just a partial flag variety of certain (twisted) affine
Kac-Moody group. The $K_v$-orbits on ${\Fl_v}$ are parameterized by
$\xcoch(T)_I^+$. For $\bar{\mu}\in\xcoch(T)_I^+$, let
$s_{\bar{\mu}}$ denote the point in ${\Fl_v}$ corresponding to
$\bar{\mu}$. More precisely, this point is the image of $\bar{\mu}$
under the map $\xcoch(T)_I\simeq
T(F)/\underline{T}^{\flat,0}(\calO)\to G(F)/K_v={\Fl_v}(k)$. Let
${\Fl_v}_{\bar{\mu}}$ be the corresponding Schubert variety, i.e.
the closure of $K$-orbit through $s_{\bar{\mu}}$. Then
\[\dim{\Fl_v}_{\bar{\mu}}=(2\rho,\bar{\mu}),\]
where $2\rho$ is the sum of all positive roots (for $H$), and by
definition $(2\rho,\bar{\mu})=(2\rho,\mu)$ for any lift $\mu$ of
$\bar{\mu}$ to $\xcoch(T)$. In addition,
${\Fl_v}_{\bar{\la}}\subset{\Fl_v}_{\bar{\mu}}$ if and only if
$\bar{\la}\preceq\bar{\mu}$. In this case,
$\dim{\Fl_v}_{\bar{\mu}}-\dim{\Fl_v}_{\bar{\la}}$ is an even
integer. For the proof of these facts, see \cite{Ri,Z}.

Let $\calP_v=\calP_{K_v}({\Fl_v})$ be the category of
$K_v$-equivariant perverse sheaves on ${\Fl_v}$, with coefficients
in $\overline{\bbQ}_\ell$. By the above facts, in each connected
component of $\Fl_v$, the dimensions of $K_v$-orbits have the
constant parity. Therefore, we have
\begin{lem}\label{semisimple}
$\calP_v$ is a semi-simple abelian category.
\end{lem}
\begin{proof}By the argument as in \cite[Proposition 1]{G1}, it is enough to show that the
stalks of the intersection cohomology sheaves have the parity
vanishing property. But this follows from the existence of Demazure
resolutions of Schubert varieties in ${\Fl_v}$ whose fibers have
pavings by affine spaces (for example see \cite[A.7]{G1}). More
precisely, the existence of such resolutions were constructed in
\cite[Sect. 8]{PR} for twisted affine flag varieties, and the
arguments as in \cite{H} apply in this situation to show that the
fibers have pavings by affine spaces.
\end{proof}

In \cite{Z,PZ}, a natural $\bbG_m$-action on ${\Fl_v}$ is
constructed. In the Kac-Moody setting, it is just the action of the
``rotation torus" on ${\Fl_v}$. Each Schubert cell is invariant
under this action.

\begin{cor}\label{Gm equiv}
Any $K_v$-equivariant perverse sheaf on ${\Fl_v}$ is automatically
$\bbG_m$-equivariant.
\end{cor}
\begin{proof}Clearly, the intersection complex is
$\bbG_m$-equivariant. Then the assertion follows from the
semisimplicity of $\calP_v$.
\end{proof}

\begin{ex}
In the special case when $G=H\otimes_k F$ and
$\underline{G}_v=H\otimes_k\calO$ is hyperspecial, then $\Fl_v$ is
just the usual affine Grassmannian $\Gr_H$ of $H$, and $\calP_v$ is
the Satake category of $H$, i.e., the category of $L^+H$-equivariant
perverse sheaves on $\Gr_H$. All the above facts are well-known.
\end{ex}

\section{Construction of the functor $\calZ$}\label{central functor}

We continue the notations as in the previous sections and let
$\underline{G}_v$ be a special parahoric group scheme of $G$ over
$\calO$. In \cite{Z}, a group scheme $\calG$ over $\bbA^1_k$ is
constructed such that

\begin{enumerate}
\item $\calG_\eta$ is connected reductive,
splits over a (tamely) ramified extension, where $\eta$ is the
generic point of $\bbA^1_k$;
\item For some choice of isomorphism $F_0\simeq F$, $\calG_{F_0}\simeq G$, where for a point $x\in \bbA^1_k$, $\calO_x$ denotes the completed local ring at $x$ and $F_x$ denotes the fractional field of $\calO_x$;

\item For any $y\neq 0$, $\calG_{\calO_y}$ is hyperspecial, (non-canonically) isomorphic to $H\otimes\calO_y$;

\item $\calG_{\calO_0}=\underline{G}_{v}$ under the isomorphism $\calG_{F_0}\simeq G$.
\end{enumerate}

The construction is as follows. Regard $I$ as the Galois group of
the cyclic cover $[e]:\bbG_m\to\bbG_m$ of degree $e$. Then the group
$I$ acts $H\times\bbG_m$. Namely, it acts on $H$ via the pinned
automorphism $\psi:I\to\Xi$, and on $\bbG_m$ via transport of
structures. Then
$\calG|_{\bbG_m}=(\Res_{\bbG_m/\bbG_m}(H\times\bbG_m))^I$, and
$\calG$ is the extension of it to $\bbA^1$ so that
$\calG_{\calO_0}=\underline{G}_v$.

Let $\Gr_\calG$ be the global affine Grassmannian of $\calG$, which
is an ind-scheme over $\bbA^1$ (see, for example \cite[\S 5]{PZ} for
the ind-representability of $\Gr_\calG$). Recall that it classifies
triples $(y,\calE,\beta)$ where $y$ is a point on $\bbA^1_k$,
$\calE$ is a $\calG$-torsor on $\bbA^1$ and $\beta$ is a
trivialization of this $\calG$-torsor away from $y$. Let
$[e]:\bbA^1\to\bbA^1$ be the natural extension of the cyclic cover
of $[e]:\bbG_m\to\bbG_m$. Let
$\wGr_\calG:=\Gr_\calG\times_{\bbA^1,[e]}\bbA^1$ be the base change.
Then
\[(\wGr_\calG)_0\simeq{\Fl_v},\quad\quad\wGr_\calG|_{\bbG_m}\simeq\Gr_H\times\bbG_m.\]
Since $I$ acts on $H$ via pinned automorphisms, it acts on $\Gr_H$,
still denoted by $\psi$. The following lemma is clear from the above
construction.
\begin{lem}\label{psi}
Under the isomorphism
$\Gr_\calG\times_{\bbG_m}\bbG_m\simeq\Gr_H\times\bbG_m$, the action
of $\ga\in I$ on the left hand side (via the Galois action on the
second factor) corresponds to the action of $\psi(\ga)\times\ga$ on
the right hand side.
\end{lem}
\begin{rmk}{\rm One should be warned that $\wGr_\calG\neq
\Gr_{\widetilde{\calG}}$, where $\widetilde{\calG}$ is the base
change of $\calG$ along $[e]:\bbA^1\to\bbA^1$.
 }\end{rmk}

Recall that we denote $\Sat_H$ to be the Satake category for $H$,
i.e., the category of $L^+H$-equivariant perverse sheaves on
$\Gr_H$, which is equivalent to $\Rep(H^\vee)$ via the geometric
Satake correspondence. Let denote
\begin{equation}\label{sat}
\calS: \Rep(H^\vee)\to \Sat_H
\end{equation}
be this equivalence. We define a functor
\[\calZ:\Sat_H\to \calP_v\] by taking the nearby cycles. More
precisely, let
\begin{equation}\label{central}
\calZ(\calF)=\Psi_{\wGr_\calG}(\calF\boxtimes\overline{\bbQ}_\ell[1]),
\end{equation}
where for an (ind)-scheme $\calX$ of (ind)-finite type over
$\bbA^1$, $\Psi_\calX$ denotes the usual nearby cycle functor (see
SGA 7, XIII for the definition of nearby cycles, and \cite[A.2]{G1}
for the explanation why the nearby cycles functors extend to
ind-schemes of ind-finite type). Recall that the theory of nearby
cycles provides an action of $\Gal(F^s/F)$ on the functor $\calZ$
via automorphisms, usually called the monodromy action.

\begin{lem}\label{monodromy}
The monodromy of $\calZ(\calF)$ is trivial.
\end{lem}
\begin{proof}
This follows from the fact that there is a $\bbG_m$-action on
$\wGr_\calG$ making the natural map $\wGr_\calG\to\bbA^1$ a
$\bbG_m$-equivariant morphism, where $\bbG_m$ acts on $\bbA^1$ via
natural dilatations (cf. \cite{Z,PZ}). In addition, the restriction
of this $\bbG_m$-action on $(\wGr_\calG)_0={\Fl_v}$ coincides with
the action of the ``rotation torus" on ${\Fl_v}$ as mentioned in
Lemma \ref{Gm equiv}.

Then for any $\calF\in\Sat_H$, the sheaf
$\calF\boxtimes\overline{\bbQ}_\ell[1]$ is $\bbG_m$-equivariant so
that the monodromy of the nearby cycle
$\calZ(\calF)=\Psi_{\wGr_\calG}(\calF\boxtimes\overline{\bbQ}_\ell)$
is opposite to the $\bbG_m$-monodromic action on $\calZ(\calF)$ (see
\cite[\S 7.2]{Z} for the definition of $\bbG_m$-monodromic sheaves
and $\bbG_m$-monodromic actions). By Corollary \ref{Gm equiv},
$\calZ(\calF)$ is $\bbG_m$-equivariant which exactly means that the
monodromy is trivial (\cite[\S 7.2]{Z}.
\end{proof}

\begin{rmk}{\rm A mixed characteristic analogue of this lemma also holds
(Theorem \ref{Shi}).
 }\end{rmk}

Let $i_0:(\wGr_\calG)_0\to\wGr_\calG$ be the closed embedding of the
special fiber and $j:\wGr_\calG\setminus(\wGr_\calG)_0\to
\wGr_\calG$ be the open complement.

\begin{cor}\label{mono-free1}
There is a canonical isomorphism $\calZ(\calF)\simeq
i_0^*j_{!*}(\calF\boxtimes\overline{\bbQ}_\ell)$.
\end{cor}
\begin{proof}This is standard. Since the monodromy is trivial, from the distinguished triangle
\[i_0^*j_*(\calF\boxtimes\overline{\bbQ}_\ell)\to \calZ(\calF)\stackrel{0}{\to} \calZ(\calF)\to\]
we obtain that $i_0^*j_*(\calF\boxtimes\overline{\bbQ}_\ell)$ lives
in perverse cohomological degree $0$ and $1$, and both cohomology
sheaves are isomorphic to $\calZ(\calF)$. But
$i^*_0j_{!*}(\calF\boxtimes\overline{\bbQ}_\ell)={^p\rH}^{0}i^*_0j_*(\calF\boxtimes\overline{\bbQ}_\ell)$,
where ${^p}\rH^*$ stands for the perverse cohomology.
\end{proof}

In what follows, for $\calF\in\Sat_H$, we denote
$\calF_{\bbG_m}=\calF\boxtimes\overline{\bbQ}_\ell[1]$ over
$\wGr_\calG|_{\bbG_m}$, and $\calF_{\bbA^1}=j_{!*}\calF_{\bbG_m}$.

Recall that the irreducible objects in $\Sat_H$ are the intersection
cohomology sheaves $\calS(V_\mu)$ on $\Gr_H$ where
$\mu\in\xcoch(T_H)^+=\xcoch(T)^+$. On the other hand, the
irreducible objects in $\calP_v$ are the intersection cohomology
sheaves $\IC_{\bar{\mu}}$ on ${\Fl_v}$, where
$\bar{\mu}\in\xcoch(T)_I^+$. For $\bar{\mu}\in\xcoch(T)_I^+$, let
$j_{\bar{\mu}}:\mathring{\Fl}_{v\bar{\mu}}\to{\Fl_v}$ be the
corresponding locally closed embedding of $K_v$-orbits.

\begin{lem}\label{mult}
For any $\mu\in\xcoch(T)^+$, let $\bar{\mu}$ be its image in
$\xcoch(T)_I^+$. Then
$j_{\bar{\mu}}^*\calZ(\calS(V_{\mu}))\simeq\overline{\bbQ}_\ell[(2\rho,\mu)]$.
\end{lem}
\begin{proof}
Consider $s_\mu\times \bbG_m\subset \Gr_H\times\bbG_m\subset
\wGr_\calG|_{\bbG_m}$. Since $\wGr_\calG$ is ind-proper over
$\bbA^1$, it extends to a section $\bbA^1\to \wGr_{\calG}$, still
denoted by $s_\mu$. By \cite[Proposition 3.5]{Z},
$s_\mu(0)\in{\Fl_v}$ is just the point $s_{\bar{\mu}}$, where
$\bar{\mu}$ is the image of $\mu$ under
$\xcoch(T_H)\to\xcoch(T_H)_I$.

Recall that $\calS(V_{\mu})$ is supported on $\bGr_\mu$, the
Schubert variety in $\Gr_H$ corresponding to $s_\mu$. Let
$\wGr_{\mu}\subset\wGr_\calG$ be the closure of
$\bGr_\mu\times\bbG_m\subset\wGr_\calG|_{\bbG_m}$. Then by the above
fact, ${\Fl_v}_{\bar{\mu}}$ is contained in the special fiber of
$\wGr_{\mu}$. In fact, it is proved in \cite{Z} that the special
fiber of $\wGr_{\mu}$ is ${\Fl_v}_{\bar{\mu}}$. In addition, it is
shown in \emph{loc. cit.} that the point $s_{\bar{\mu}}$ is smooth
in $\wGr_{\mu}$. The lemma then is clear.
\end{proof}

The following key result is established in \cite{G1} the split case
and in \cite[Theorem 7.3]{Z} in general. Let $D_{K_v}({\Fl_v})$ be
the $K_v$-equivariant derived category on $\Fl_v$, and
$\star:\calP_v\times\calP_v\to D_{K_v}({\Fl_v})$ be the convolution
product functor. For a precise definition of the convolution
product, see for example \cite{MV,Z}.
\begin{prop}\label{comm3}
For any $\calF_1\in\Sat_H$ and $\calF_2\in\calP_v$, there is a
canonical isomorphism
$\calZ(\calF_1)\star\calF_2\simeq\calF_2\star\calZ(\calF_1)$ and
both are objects in $\calP_v$.
\end{prop}

Let us briefly review the proof. Let us define the
Beilinson-Drinfeld Grassmannian
\begin{equation}\label{BD Grass}
\Gr^{BD}_\calG(R)=\left\{(y,\calE,\beta)\ \left|\ \begin{split}&y\in
\bbA^1(R), \calE \mbox{ is a } \calG\mbox{-torsor on
}\bbA^1_R,\mbox{ and}\\& \beta:\calE|_{(\bbG_m)_R-\Ga_y}\simeq
\calE^0|_{(\bbG_m)_R-\Ga_y} \mbox{ is a trivialization}
\end{split}\right.\right\},
\end{equation}
where $\calE^0$ denotes the trivial torsor. We have
\[\Gr_\calG^{BD}|_{\bbG_m}\simeq\Gr_{\calG}|_{\bbG_m}\times (\Gr_\calG)_0,\quad (\Gr_{\calG}^{BD})_0\simeq(\Gr_\calG)_0\simeq{\Fl_v}.\]
As before, we denote $\wGr_{\calG}^{BD}$ to be the base change along
$[e]:\bbA^1\to\bbA^1$. Over $\wGr_{\calG}^{BD}|_{\bbG_m}\simeq
\wGr_{\calG}|_{\bbG_m}\times (\Gr_{\calG})_0$, we can form the
external product $(\calF_1)_{\bbG_m}\boxtimes\calF_2$. Then the
isomorphism in the proposition is induced from the canonical
isomorphisms (cf. \cite[Theorem 7.3]{Z})
\[\calZ(\calF_1)\star\calF_2\simeq\Psi_{\wGr^{BD}_\calG}((\calF_1)_{\bbG_m}\boxtimes\calF_2)\simeq \calF_2\star\calZ(\calF_1).\]

Again, since the monodromy of
$\Psi_{\wGr^{BD}_\calG}((\calF_1)_{\bbG_m}\boxtimes\calF_2)$ is
trivial, we have
\begin{equation}\label{mono-free2}\calZ(\calF_1)\star\calF_2\simeq
i_0^*j_{!*}((\calF_1)_{\bbG_m}\boxtimes\calF_2)\simeq
\calF_2\star\calZ(\calF_1),
\end{equation}
where $i_0,j$ are corresponding closed and open embedding.

\begin{cor}
The convolution of $\calP_v$ is bi-exact. Therefore, $\calP_v$ is a
monoidal category.
\end{cor}
\begin{proof}
Observe that $\calP_v$ is semi-simple (Lemma \ref{semisimple}) and
every irreducible object in $\calP_v$ is a direct summand of some
$\calZ(\calF)$. Indeed, for $\bar{\mu}\in\xcoch(T)_I^+$, let $\mu$
be a lift of it in $\xcoch(T)^+$. Then by Lemma \ref{mult},
$\IC_\mu$ appears as a direct summand of $\calZ(\calS(V_\mu))$ with
multiplicity one. As the convolution (for left and right) with
$\calZ(\calS(V_\mu))$ is exact, the convolution with its direct
summand is also exact. The first claim follows.

It is well-known that the convolution functor $D_{K_v}(\Fl_v)\times
D_{K_v}(\Fl_v)\to D_{K_v}(\Fl_v)$ is monoidal. Its restriction to
$\calP_v\times\calP_v$ takes value in $\calP_v$. As $\calP_v$ is a
full subcategory of $D_{K_v}(\Fl_v)$, the associativity constraints
are morphisms in $\calP_v$. Therefore, $\calP_v$ is a monoidal
subcategory of $D_{K_v}(\Fl_v)$.
\end{proof}

\begin{rmk}{\rm
(i) According to \cite[Remark 4.5]{MV}, the exactness of the
convolution product probably would imply that $LG\times^K{\Fl_v}\to
{\Fl_v}$ is (stratified) semi-small.

(ii) By the same argument, the convolution bi-functor $\star:
\calP({\Fl_v})\times\calP_v\to\calP({\Fl_v})$ is also exact, where
$\calP({\Fl_v})$ is the category of perverse sheaves on $\Fl_v$.
 }\end{rmk}

\section{$\calZ$ is a central functor}\label{central property}
In this section, we show that $\calZ$ is a central functor in the
sense of \cite{B}. By Lemma \ref{mult}, together with some general
nonsense, this already implies that $\calP_v$ is equivalent to
$\Rep(\tilde{G}^\vee)$ for some closed subgroup
$\tilde{G}^\vee\subset H^\vee$. In the next section, we will
identify $\tilde{G}^\vee$ explicitly. We will also determine a fiber
functor of $\calP_v$.

\begin{thm-def}The functor $\calZ:\Sat_H\to\calP_v$ is naturally a monoidal functor.
\end{thm-def}
\begin{proof}The proof is literally the same as the proof in
\cite[Theorem 1(c)]{G1}. We repeat the argument here in order to
make the definition of this monoidal structure explicit.

Let $\Gr_\calG\underset{\bbA^1}{\tilde{\times}}\Gr_\calG$ be the
ind-scheme over $\bbA^1$ classifying
\begin{equation}
\Gr_\calG\underset{\bbA^1}{\tilde{\times}}\Gr_\calG(R)=\left\{(y,\calE,\calE',\beta,\beta')\
\left|\ \begin{split} &y\in \bbA^1(R), \calE,\calE' \mbox{ are two }
\calG\mbox{-torsors}\\ &\mbox{ on } \bbA^1_R,
\beta:\calE|_{\bbA^1_R-\Ga_y}\simeq \calE^0|_{\bbA^1_R-\Ga_y} \mbox{
is a }\\& \mbox{trivialization},
\beta':\calE'|_{\bbA^1_R-\Ga_y}\simeq\calE|_{\bbA^1_R-\Ga_y}\end{split}\right.\right\}.
\end{equation}
The notation suggests that
$\Gr_\calG\underset{\bbA^1}{\tilde{\times}}\Gr_\calG$ is a kind of
twisted product. Indeed, let $\calL^+\calG$ be the global jet group
of $\calG$, which classifies a point on $\bbA^1$ and a
trivialization of the trivial $\calG$-torsor over the formal
neighborhood of this point (cf. \cite[\S 3.1]{Z}). Then
$\calL^+\calG$ naturally acts on $\Gr_\calG$. In addition, there is
a $\calL^+\calG$-torsor over $\Gr_\calG$ classifying quadruples
$(y,\calE,\beta,\ga)$, where the triple $(y,\calE,\beta)$ is as in
the definition of $\Gr_\calG$ and $\ga$ is a trivialization of
$\calE$ over the formal neighborhood of $y$ (this is indeed the
global loop group $\calL\calG$ of $\calG$ introduced in \cite[\S
3.1]{Z}). Then
$\Gr_\calG\underset{\bbA^1}{\tilde{\times}}\Gr_\calG\simeq\calL\calG\overset{\calL^+\calG}{\times}\Gr_\calG$.
Let us denote the base change of the this isomorphism along
$[e]:\bbA^1\to\bbA^1$ by
$\wGr_\calG\underset{\bbA^1}{\tilde{\times}}\wGr_\calG\simeq\widetilde{\calL\calG}\overset{\widetilde{\calL^+\calG}}{\times}\wGr_\calG$.

Let $\calF_1$ and $\calF_2$ be two objects in $\Sat_H$. We can form
the twisted product
$(\calF_1)_{\bbG_m}\tilde{\times}(\calF_2)_{\bbG_m}$ over
$\wGr_\calG|_{\bbG_m}\underset{\bbG_m}{\tilde{\times}}\wGr_\calG|_{\bbG_m}$.
We claim that there is a canonical isomorphism
\[\Psi_{\wGr_\calG\underset{\bbA^1}{\tilde{\times}}\wGr_\calG}((\calF_1)_{\bbG_m}\tilde{\times}(\calF_2)_{\bbG_m})\simeq\calZ(\calF_1)\tilde{\times}\calZ(\calF_2).\]
Indeed, let $V_i\subset\wGr_\calG$ be the closure of the support of
$\calF_i\boxtimes\overline{\bbQ}_\ell[1]$ in $\wGr_\calG|_{\bbG_m}$.
Let $\calL^+_n\calG$ be the $n$th jet group such that the action of
$\calL^+\calG$ on $V_i$ factors through $\calL^+_n\calG$. The
corresponding $\calL^+_n\calG$ torsor over $\Gr_\calG$ is denoted by
$\calL_n\calG$. Let us denote $\widetilde{\calL^+_n\calG}$ and
$\widetilde{\calL_n\calG}$ be their base changes along $[e]$. Then
one can check the isomorphism  after pullback along
$\widetilde{\calL_n\calG}\times_{\bbA^1}V_2\to\widetilde{\calL_n\calG}\overset{\widetilde{\calL^+_n\calG}}{\times}V_2$,
and the isomorphism follows from \cite[Theorem 5.2.1]{G1}.

Now
$\Gr_\calG\underset{\bbA^1}{\tilde{\times}}\Gr_\calG\to\Gr_\calG,
(y,\calE,\calE',\beta,\beta)\mapsto (y,\calE',\beta\beta')$ is
ind-proper and taking nearby cycles commutes with proper
push-forward. Therefore we obtain the canonical isomorphism
\[\calZ(\calF_1\star\calF_2)\simeq\calZ(\calF_1)\star\calZ(\calF_2)\]
In addition, working over
$\Gr_\calG\underset{\bbA^1}{\tilde{\times}}\Gr_\calG\underset{\bbA^1}{\tilde{\times}}\Gr_\calG$,
one can see that this isomorphism makes $\calZ$ a monoidal functor.
\end{proof}

Let us recall the definition of central functors as in \cite{B}.
Namely, if $F:\calC\to\calD$ is a monoidal functor between two
monoidal categories and assume that $\calC$ is a symmetric monoidal
category, then $F$ (together with the following data) is called
central if
\begin{enumerate}
\item there is an isomorphism $c$ of the bi-functors
$\calC\times\calD\to\calD, (X,Y)\mapsto F(X)\otimes Y$ and
$(X,Y)\mapsto Y\otimes F(X)$, i.e. an isomorphism
$c_{X,Y}:F(X)\otimes Y\simeq Y\otimes F(X)$ functorial in $X,Y$;

\item for $X,X'\in\calC$, the following diagram is commutative
\[\begin{CD}
F(X)\otimes F(X')@>c_{X,F(X')}>> F(X')\otimes F(X)\\
@VVV@VVV\\
F(X\otimes X')@>F(\sigma_{X,X'})>>F(X'\otimes X),
\end{CD}\]
where $\sigma$ is the commutativity constraint of $\calC$;

\item for $X\in\calC$ and $Y,Y'\in\calD$, the following diagram is
commutative
\[\begin{CD}
F(X)\otimes Y\otimes Y'@>c_{X,Y}\otimes\id>> Y\otimes F(X)\otimes
Y'\\
@V c_{X,Y\otimes Y'}VV @VV\id\otimes c_{X,Y'}V\\
 Y\otimes Y'\otimes F(X)@=Y\otimes Y'\otimes F(X);
\end{CD}\]

\item for $X,X'\in\calC, Y\in\calD$, the following diagram is
commutative
\[\begin{CD}
F(X)\otimes F(X')\otimes Y@>\id\otimes c_{X',Y}>>F(X)\otimes
Y\otimes F(X')@>c_{X,Y}\otimes\id>>Y\otimes F(X)\otimes F(X')\\
@VVV@.@VVV\\
F(X\otimes X')\otimes Y@>c_{X\otimes X',Y}>>Y\otimes F(X\otimes
X')@=Y\otimes F(X\otimes X').
\end{CD}\]
\end{enumerate}

\begin{prop}
The functor $\calZ$, together with the canonical isomorphisms
provided in Proposition \ref{comm3}, is a central functor. There is
an algebraic group $\tilde{G}^\vee\subset H^\vee$ together with an
equivalence $S:\calP_v\simeq\Rep(\tilde{G}^\vee)$ such that
$S\circ\calZ\simeq\Res_{H^\vee}^{\tilde{G}^\vee}$ as tensor
functors, where $\Res_{H^\vee}^{\tilde{G}^\vee}$ is the restriction
functor from $\Rep(H^\vee)$ to $\Rep(\tilde{G}^\vee)$.
\end{prop}
\begin{proof}
Since every object in $\calP_v$ appears as a direct summand of some
object in the essential image of $\calZ$, the second statement of
the proposition is a direct consequence of the first statement and
Proposition 1 of \cite{B}.

The first statement can be checked literally the same as in the
\cite{G2}. In fact, in this case, the proof is even simpler. Namely,
conditions (3) and (4) are checked as the same way as in \emph{loc.
cit}. To check condition (2), observe that the monodromy of all the
nearby cycles involve is trivial. They the nearby cycles can be
expressed via intermediate extensions as in \ref{mono-free1} and
\ref{mono-free2}, rather than via the homotopy (co)limits of certain
ind-pro system of sheaves as in \emph{loc. cit}.
\end{proof}

Now we would like to endow $\calP_v$ with a fiber functor. We begin
with the following general lemma.
\begin{lem}Let $G_1\subset G_2$ be a closed embedding of affine algebraic
groups over a field $E$ (of characteristic zero). Let
$F:\Rep(G_1)\to\on{Vect}_{E}$ be an $E$-linear exact and faithful
functor. Assume that: (i) $F(X\otimes Y)$ and $F(X)\otimes F(Y)$ are
(non-canonically) isomorphic; (ii) $F\circ\Res_{G_2}^{G_1}$ is a
fiber functor of $\Rep(G_2)$. Then $F$ has a unique fiber functor
structure which induces the fiber functor structure of
$F\circ\Res_{G_2}^{G_1}$ as in (ii).
\end{lem}
\begin{rmk}{\rm We are not sure whether the first assumption is necessary.
 }\end{rmk}
\begin{proof}The uniqueness is clear.
We write $R=\Res_{G_2}^{G_1}$ for simplicity. For any
$X\in\Rep(G_2)$, let $\langle X\rangle$ denote the full subcategory
of $\Rep(G_2)$ consisting the objects that are isomorphic to
subquotients of $X^n, n\in\bbN$, and $\langle R(X)\rangle$ denote
the full subcategory of $\Rep(G_1)$ consisting of the objects that
are isomorphic to subquotients of $R(X)^n, n\in\bbN$. Let us denote
$\End(FR|_{\langle X\rangle})$ (resp. $\End(F|_{\langle
R(X)\rangle})$) the endomorphism algebra of the restriction of the
functor $FR$ (resp. $F$) to $\langle X\rangle$ (resp. $\langle
R(X)\rangle$). They are finite dimensional $E$-algebras and clearly,
the $E$-algebra homomorphism $\End(F|_{\langle R(X)\rangle})\to
\End(FR|_{\langle X\rangle})$ is injective. According to \cite[Lemma
2.13]{DM}, there are canonical equivalences and the following
commutative diagram
\[\begin{CD}
\langle X \rangle @>{a_X}>\simeq> \End(FR|_{\langle X
\rangle})\Mod @>\omega>>\on{Vect}_E\\
@VR VV@VVV@|\\
\langle R(X)\rangle @>b_{R(X)}>\simeq> \End(F|_{\langle
R(X)\rangle})\Mod @>\omega>>\on{Vect}_E
\end{CD}.\]
In addition, $\omega a_X\simeq FR$ and $\omega b_{R(X)}\simeq F$.
Observe that if $\langle X\rangle $ is a subcategory of $\langle Y
\rangle$. Then we have a natural algebra homomorphism
$\End(FR|_{\langle Y \rangle})\to \End (FR|_{\langle X \rangle})$.
Then $A=\underrightarrow{\lim}_{X\in\Rep G_2}\End(FR|_{\langle X
\rangle})^\vee$ is a coalgebra. Similarly, we can define
$B=\underrightarrow{\lim}_{X\in\Rep G_2}\End(F|_{\langle
R(X)\rangle})^\vee$. We have the surjective map of coalgebras $A\to
B$, and
\[\begin{CD}
\Rep G_2 @>{a}>\simeq> A\mbox{-Comod} @>\omega>>\on{Vect}_E\\
@V R VV@VVV@|\\
\Rep G_1 @>b>\simeq> B\mbox{-Comod} @>\omega>>\on{Vect}_E
\end{CD}.\]
By the assumption (i) and \cite[Proposition 2.16]{DM}, the tensor
structures on $\Rep G_1$ and $\Rep G_2$ induces $B\otimes B\to B$
and $A\otimes A\to A$ respectively. Since the restriction functor
$R$ is a tensor functor, we have the commutative diagram
\[\begin{CD}
A\otimes A@>>>A\\
@VVV@VVV\\
B\otimes B@>>>B
\end{CD}.\]
By assumption (ii), $\omega a\simeq FR$ is a fiber functor of $\Rep
G_2$, and therefore we know that $A=\calO_{G_2}$ and the map
$A\otimes A\to A$ is the usual multiplication. Since the map $A\to
B$ is surjective, this implies that the map $B\otimes B\to B$ is
also associative and commutative. By \cite[Proposition 2.16]{DM}
again, this implies that the functor $F$ respect to the
associativity and the commutativity constraints. The lemma follows.
\end{proof}
\begin{cor}\label{fib}
The functor given by taking the cohomology
$\rH^*:\calP_v\to\on{Vect}_{\overline{\bbQ}_\ell}$ has a natural
structure as a fiber functor.
\end{cor}
\begin{proof}It is well-known (e.g. from the decomposition theorem) that there exists an isomorphism
$\rH^*({\Fl_v},\calF_1\star\calF_2)\simeq
\rH^*({\Fl_v},\calF_1)\otimes \rH^*({\Fl_v},\calF_2)$
(non-canonically). Since taking nearby cycles commutes with proper
push forward, we have a canonical isomorphism $\rH^*\circ\calZ\simeq
\rH^*$. Since $\rH^*:\Sat_H\to\on{Vect}_{\overline{\bbQ}_\ell}$ is a
fiber functor, the assertion follows from the above lemma.
\end{proof}

\section{Identification of the group $\tilde{G}^\vee$ with
$(H^\vee)^I$}\label{proof} We need to describe the group
$\tilde{G}^\vee$ from the last section. To begin with, let us review
how the geometric Satake correspondence (together with a choice of
an ample line bundle on $\Gr_H$) gives rise to a pinned group
$(H^\vee,B_H^\vee,T_H^\vee,X^\vee)$. First, once we choose
$T_H\subset B_H\subset H$, the construction of \cite{MV} provides us
$T_H^\vee\subset B_H^\vee\subset H^\vee$. Namely, let $U_H\subset
B_H$ be its unipotent radical. For $\mu\in\xcoch(T_H)$, let $S_\mu$
be the semi-infinite orbit on $\Gr_H$ passing through $s_\mu$ as
introduced in \cite{MV} (i.e., the $LU_H$-orbit passing through
$s_\mu$). Let $S_{\leq\mu}=\cup_{\la\preceq\mu}S_\la$ and
$S_{<\mu}=\cup_{\la\prec\mu}S_\la$. Then the fiber functor
$\rH^*:\Sat_H\to\on{Vect}_{\bbQ_\ell}$ has a canonical filtration
(called the \emph{MV filtration}) given by $\ker(\rH^*(\Gr_H,-)\to
\rH^*(S_{<\mu},-))$. This defines a Borel $B_H^\vee\subset H^\vee$.
In addition, it is proved that the filtration admits a canonical
splitting, i.e. a canonical isomorphism $\rH^*(\Gr_H,-)\simeq
\bigoplus_{\mu} \rH^*_c(S_\mu,-)$. This provides a maximal torus
$T_H^\vee\subset B_H^\vee$. Let $\calL$ be an ample line bundle on
$\Gr_H$, and let $c(\calL)\in \rH^2(\Gr_H,\overline{\bbQ}_\ell)$ be
its Chern class. Then it is shown in \cite{Gi,YZ} that the cup
product with this class realizes $c(\calL)$ as a principal nilpotent
element in $X^\vee=\frakh^\vee=\Lie H^\vee$. In addition, by
\cite[Proposition 5.6]{YZ}, the quadruple
$(H^\vee,B_H^\vee,T_H^\vee,X^\vee)$ is indeed a pinned reductive
group.

\begin{rmk}{\rm One remark is in order. In \cite{YZ}, all the assertions are proved for the affine Grassmannian
defined over $\bbC$. The only place where the complex topology is
used, besides the issue of dealing with $\bbZ$-coefficients as in
\cite{MV}, is to define the coproduct on $\rH^*(\Gr_H,\bbZ)$ by
realizing $\Gr_H$ as being homotopic to the based loop space of a
maximal compact subgroup of $H_\bbC$. However, one can provide a
commutative and cocommutative Hopf algebra structure on
$\rH^*(\Gr_H,\bbZ)$ using the Beilinson-Drinfeld Grassmannian. More
precisely, one can use the isomorphism (2.11) in \emph{loc. cit.} to
define the comultiplication map by the formula (2.12) in \emph{loc.
cit.}. This map on the other hand can be realized as follows. There
is the Beilinson-Drinfeld Grassmannian $\pi:\Gr_2\to\bbA^2$ whose
fiber over a point in the diagonal $\Delta\subset\bbA^2$ is $\Gr_H$
and whose fiber over a point off the diagonal is $\Gr_H\times\Gr_H$
(cf. \cite[Sect. 5]{MV}). Then $R^i\pi_*\bbQ_\ell$ is a
constructible sheaf on $\bbA^2$, constant along the stratification
$\bbA^2=\Delta\cup(\bbA^2-\Delta)$. Now the usual cospecialization
map of constructible sheaves gives rise to the comultiplication.
From this latter definition, the usual arguments for the
commutativity constraints as in \cite{MV} show that this defined
comultiplication is indeed cocommutative. The proof of \cite[Lemma
5.1]{YZ} that $c^{T_H}(\calL)$ is primitive under this Hopf algebra
structure can be replaced by the following argument: as is
well-known (e.g see \cite[1.1.9]{Z1}), if $\calL$ is ample on
$\Gr_H$, then there is an ample line bundle on $\Gr_2$, which away
from the diagonal is $\calL\boxtimes\calL$ and on the diagonal is
$\calL$. Now the above arguments and all the remaining arguments of
\cite{YZ} apply to the situation when $\Gr_H$ is defined over
arbitrary field $k$ and sheaves have
$\overline{\bbQ}_\ell$-coefficients.
 }\end{rmk}

\begin{rmk}\label{independence}{\rm
As explained in \cite[Theorem 3.6]{MV}, the above pinning
$(H^\vee,B_H^\vee,T_H^\vee,X^\vee)$ is in fact independent of the
choice of $T\subset B$. Another way to deduce this fact is as
follows. The natural grading on the cohomological functor $\rH^*$
defines a one-parameter subgroup $\bbG_m\to H^\vee$ and $T_H^\vee$
is just the centralizer of this subgroup, which is independent of
the choice of $T\subset B$. On the other hand, $B^\vee_H$ is
completely determined by $X^\vee$, which is also independent of the
choice of $T\subset B$. In other words, there is a canonical
morphism from the Lefschetz $\SL_2$ to $H^\vee$, which gives a
principal $\SL_2$ in $H^\vee$, and the pinning is determined by this
principal $\SL_2$.
 }\end{rmk}

Recall that we denote $\psi$ to be the action of $I$ on $\Gr_H$. The
action of $\ga\in I$ will map $\bGr_\la$ isomorphically to
$\bGr_{\ga(\la)}$. Therefore $\ga_*: D(\Gr_H)\to D(\Gr_H)$ naturally
gives rise to $\ga_*:\Sat_H\to\Sat_H$ for $\ga\in\Ga$. In this way,
$I$ acts on $\Sat_H$ via tensor automorphisms. Under the geometric
Satake correspondence, $I$ acts on $H^\vee$ clearly as pinned
automorphisms with respect to the pinning we mentioned above.

\begin{thm}
$\tilde{G}^\vee\simeq (H^\vee)^I$.
\end{thm}
\begin{rmk}\label{non-connected}{\rm
Observe that $(H^\vee)^I$ is not necessarily a connected reductive
group. For example: let $H^\vee=\GL_{2n+1}$, let $J$ be the matrix
with $1s$ on the anti-diagonal and $0s$ elsewhere. Let
$\Ga=\{1,\ga\}$ and $\ga$ acts on $H^\vee$ via $g\mapsto
J(g^t)^{-1}J$. Then $(H^\vee)^I=\on{O}_{2n+1}$.
 }\end{rmk}
\begin{proof}
Since $\calP_v$ is semi-simple, $\tilde{G}^\vee$ is a reductive
subgroup of $H^\vee$. We first see that $\tilde{G}^\vee\subset
(H^\vee)^I$. The following lemma is a direct consequence of
\cite[Corollary 2.9]{DM}.

\begin{lem}Let $f:H_2\to H_1$ be a
homomorphism of algebraic groups and let $\omega^f$ denote the
induced tensor functor $\Rep(H_1)\to\Rep(H_2)$ (if $f$ is a closed
embedding then $\omega^f$ is the restriction functor
$\Res_{H_1}^{H_2}$). Let $I\subset\Aut(H_1)$ so that it acts on
$\Rep(H_1)$ via tensor automorphisms. If for any $\ga\in I$,
$\omega^f\circ\omega^\ga\simeq \omega^f$, then $f$ factors through
$f:H_2\to H_1^I\subset H_1$.
\end{lem}

Now, $I$ acts on $\wGr_\calG=\Gr_\calG\times_{\bbA^1}\bbA^1$ via the
action on the second factor $\bbA^1$ by deck transformations. By
Lemma \ref{psi}, we have
\[\calZ(\ga_*\calF)=\Psi_{\wGr_\calG}(\ga_*\calF\boxtimes\overline{\bbQ}_\ell[1])\simeq\Psi_{\wGr_\calG}((\psi(\ga)\times\ga)_*(\calF\boxtimes\overline{\bbQ}_\ell[1]))\simeq\Psi_{\wGr_\calG}(\calF\boxtimes\overline{\bbQ}_\ell[1])=\calZ(\calF).\]
In other words, we have the tensor isomorphism between
$\calZ\circ\ga_*$ and $\calZ$ for all $\ga\in I$. From the above
lemma, $\tilde{G}^\vee\subset(H^\vee)^I$.

Therefore, we have successive restriction functors
\[\Rep(H^\vee)\to\Rep((H^\vee)^I)\to\Rep(\tilde{G}^\vee).\] To prove
that $\tilde{G}^\vee=(H^\vee)^I$, it is enough to show that the
above restriction induces an isomorphism of $K$-groups
$K(\Rep(H^\vee)^I)\simeq K(\Rep(\tilde{G}^\vee))$.

As the group $(H^\vee)^I$ may not be connected, we need to be
careful to describe its representation ring.

Let $(H^\vee)^{I,0}$ denote the neutral connected component of
$(H^\vee)^I$. This is a connected reductive group with maximal torus
$(T^\vee)^{I,0}$, the neutral connected component of $(T^\vee)^I$.
The key fact is the following lemma.
\begin{lem}\label{connected components}
The natural map
\[(T^\vee)^{I}/(T^\vee)^{I,0}\to (H^\vee)^{I}/(H^\vee)^{I,0}\]
is an isomorphism.
\end{lem}
\begin{proof} We will not distinguish a group from its
$E$-points. As we purely work with dual groups, we switch the
notation $H^\vee$ to $H$ etc. in the proof. Let us choose $\gamma$
to be a generator of $I$. We need to show that
$\pi_0(T^I)=\pi_0(H^I)$. Let $N$ be the normalizer of $T$ in $H$ and
let $W=N/T$ be the Weyl group. Then $I$ acts on $W$ naturally. Let
us define a right action of $W^I$ on $\rH^1(I,T)$ as follows.
Suppose $w$ is in $W^I$ and $c$ be a cohomology class; lift $w$ to
$n\in N$ and lift $c$ to a cocycle $\varphi:I\to T$. Then we set
$(c\cdot w) (\gamma)=[n^{-1}\varphi(\gamma)\gamma(n)]$. It is clear
that this is independent of all choices. We will deduce Lemma
\ref{connected components} from the following fact.

\begin{lem}Under the above definition, every element $w\in W^I$ acts on $\rH^1(I,T)$
via a group automorphism. In addition, $\rH^1(I,H)$ is the quotient
of $\rH^1(I,T)$ via the above action.
\end{lem}
\begin{proof}Observe that the map $N^I\to
W^I$ is surjective. Indeed, let $H_\der$ be the derived group of $H$
and $H_\s$ be the simply-connected cover of $H_\der$. We have
corresponding groups $N_\der,N_\s, T_\der, T_\s$. We now apply the
argument of \cite{St} p. 55 (5) to $H_{\s}$ and $\gamma$. Our
assumption that $\gamma$ is pinned allows us to take $t=1$ in
\emph{loc. cit.}. It follows that the natural map $N_{\s}^I\to W^I$
is surjective; therefore the same is true for $N^I\to W^I$ (another
argument of this surjectivity can be found in \cite[Lemma 6.2]{Bo}).
By taking the lift $w$ to $n\in N^I$, it is clear that $w$ acts on
$\rH^1(I,T)$ via group automorphisms. The second statement was
proved in \cite{PZ}.
\end{proof}
\begin{rmk}{\rm The above lemma in particular shows that if $I$ acts on $H$ via pinned automorphisms, then
$\rH^1(I,H)$ has a canonical abelian group structure. This does not
necessarily hold for arbitrary action of $I$.
 }\end{rmk}
\begin{cor}\label{cor}
The preimage of $1\in \rH^1(I,H)$ under $\rH^1(I,T)\to\rH^1(I,H)$ is
$1$.
\end{cor}

We continue to prove Lemma \ref{connected components}. First, if $H$
is simply-connected, then $H^I$ is connected as is shown in
\cite[Theorem 8.2]{St}.  On the other hand, $I$ acts on $T$ via
permuting a basis of $\xch(T)$. Therefore, $T^I$ is also connected.
The lemma holds in this case. For general $H$, let $H_\der$ be the
derived group of $H$ and $H_\s$ be the simply-connected cover of
$H_\der$. Let $T_\der$ and $T_\s$ be the corresponding preimages of
$T$. Write
$$
1\to Z\to H_{\s}\to H_{\der}\to  1
$$
which then gives
\begin{equation}\label{ext1}
1   \to \pi_0(H^I_{\der})\to \rH^1(I, Z)\to \rH^1(I, H_{\s}) .
\end{equation}
Similarly, the sequence of maximal tori
$$
1\to Z\to T_{\s}\to T_{\der}\to 1
$$
gives
\begin{equation}\label{ext2}
1\to \pi_0(T^I_{\der})\to \rH^1(I, Z) \to \rH^1(I, T_{\s}).
\end{equation}
Comparing (\ref{ext1}) and (\ref{ext2}) and applying Corollary
\ref{cor}, we obtain that the natural map
\begin{equation}\label{ext3}
\pi_0(T^I_{\der})\xrightarrow{\sim} \pi_0(H^I_{\der})
\end{equation}
is an isomorphism. Now consider
$$
1\to H_{\der}\to H\to D\to 1
$$
which gives
\begin{equation}\label{ext4}
1\to \pi_0(H^I_{\der})\to \pi_0(H^I)\to\pi_0(D^I)\to \rH^1(I,
H_{\der}) .
\end{equation}
Similarly, the sequence of maximal tori
$$
1\to T_{\der}\to T\to D\to 1
$$
give
\begin{equation}\label{ext5}
1\to \pi_0(T^I_{\der})\to \pi_0(T^I)\to\pi_0(D^I)\to \rH^1(I,
T_{\der}) .
\end{equation}
Comparing (\ref{ext4}) and (\ref{ext5}) and using Corollary
\ref{cor} again, we obtain that the natural map
\begin{equation}\label{ext6}
\pi_0(T^I)\xrightarrow{\sim} \pi_0(H^I)
\end{equation}
is an isomorphism.
\end{proof}

\quash{ Let $N_{H^I}(T^I)$ (resp. $N_{H^{I,0}}(T^{I,0})$) the the
normalizer of $T^I$ (resp. $T^{I,0}$) in $H^I$ (resp. $H^{I,0}$). As
$Z_H(T^I)=T$, we have $N_{H^I}(T^I)=N^I$ and
$N_{H^{I,0}}(T^{I,0})=N_{H^I}(T^I)\cap H^{I,0}$. Combining with (the
proof of) Lemma \ref{connected components}, we have
\[\frac{N_{H^{I,0}}(T^{I,0})}{T^{I,0}}\simeq \frac{N_{H^I}(T^I)}{T^I}\simeq W^I.\]
}

Recall that there is a natural partial order ``$\preceq$" on
$\xcoch(T)_I$ given by \eqref{order}. We claim that

\begin{lem}\label{irr}
(i) For $\bar{\mu}\in\xcoch(T)_I^+$, there is a unique (up to
isomorphism) irreducible representation $W_{\bar{\mu}}$ of
$(H^\vee)^I$ of highest weight $\bar{\mu}$. In addition, any
irreducible representation of $(H^\vee)^I$ is of this form.

(ii) The multiplicity of the $\bar{\mu}$-weight in $W_{\bar{\mu}}$
is one.
\end{lem}
\begin{proof}
Indeed, let $\bar{\bar{\mu}}$ be the image of $\bar{\mu}$ in
$\xcoch(T)_I/\xcoch(T)_{I,\on{tor}}$, and $W_{\bar{\bar{\mu}}}$ be
the unique irreducible representation of $(H^\vee)^{I,0}$ of highest
weight $\bar{\bar{\mu}}$. Then by Lemma \ref{connected components}
and the Frobenius reciprocity
\[\ind_{(H^\vee)^{I,0}}^{(H^\vee)^I}W_{\bar{\bar{\mu}}}\simeq\bigoplus_{\chi\in \Rep((T^\vee)^I/(T^\vee)^{I,0})}W\otimes\chi,\]
where $W$ is an irreducible representation of $(H^\vee)^I$, whose
restriction to $(H^\vee)^{I,0}$ is isomorphic to
$W_{\bar{\bar{\mu}}}$. It is then clear than exact one of
($W\otimes\chi$) is an irreducible representation of $(H^\vee)^I$ of
highest weight $\bar{\mu}$. This proves the existence. Then
uniqueness is also clear because by the Frobenius reciprocity, every
irreducible representation of $(H^\vee)^I$ appears as a direct
summand in $\ind_{(H^\vee)^{I,0}}^{(H^\vee)^I}W_{\bar{\bar{\mu}}}$
for some $\bar{\bar{\mu}}\in\xcoch(T)_I/\xcoch(T)_{I,\on{tor}}$.
\end{proof}

Now we finish the proof of the theorem. Let $\mu\in\xcoch(T)^+$ be a
lift of $\bar{\mu}$. Then,
$\Res_{H^\vee}^{(H^\vee)^I}[V_{\mu}]=[W_{\bar{\mu}}]+\sum_{\bar{\la}\prec\bar{\mu}}c_{\bar{\la}\bar{\mu}}[W_{\bar{\la}}]$,
where $[X]$ stands for the element in the $K$-group corresponding to
$X$. Therefore,
\[\Res_{H^\vee}^{\tilde{G}^\vee}[V_\mu]=\Res_{(H^\vee)^I}^{\tilde{G}^\vee}[W_{\bar{\mu}}]+\sum_{\bar{\la}\prec\bar{\mu}}c_{\bar{\la}\bar{\mu}}\Res_{(H^\vee)^I}^{\tilde{G}^\vee}[W_{\bar{\la}}].\]
On the other hand, for $\bar{\mu}\in\xcoch(T)_I^+$, the intersection
cohomology sheaf $\IC_{\bar{\mu}}\in\calP_v$ gives rise to an
irreducible object $U_{\bar{\mu}}$ in $\Rep(\tilde{G}^\vee)$. By
Lemma \ref{mult}, we have
\[\Res_{H^\vee}^{\tilde{G}^\vee}[V_{\mu}]=[U_{\bar{\mu}}]+\sum_{\bar{\la}\prec\bar{\mu}}d_{\bar{\la}\bar{\mu}}[U_{\bar{\la}}].\]
By induction on $\bar{\mu}$, one immediately obtains that
\[\Res_{(H^\vee)^I}^{\tilde{G}^\vee}[W_{\bar{\mu}}]=U_{\bar{\mu}}+\sum_{\bar{\la}\prec\bar{\mu}}e_{\bar{\la}\bar{\mu}}[U_{\bar{\la}}].\]
Since $[W_{\bar{\mu}}]$ (resp. $[U_{\bar{\mu}}]$) form a
$\bbZ$-basis of $K(\Rep((H^\vee)^I))$ (resp.
$K(\Rep(\tilde{G}^\vee))$), this implies that
$\Res_{(H^\vee)^I}^{\tilde{G}^\vee}$ is an isomorphism and therefore
$\tilde{G}^\vee=(H^\vee)^I$.
\end{proof}
\quash{
\begin{rmk}\label{connected}{\rm
Although $(H^\vee)^I$ is not necessarily connected, in many aspect,
it behaves like the connected reductive group due to Lemma
\ref{connected components}. One example is provided by Lemma
\ref{irr}. Let us list a few more. We can talk about the root datum
of $(H^\vee)^I$. Namely, the set of roots for $(H^{I,0},T^{I,0})$,
denoted by
\[\Phi(H^{I,0},T^{I,0})\subset
\xch(T^{I,0})=\xch(T)_I/\xch(T)_{I,\on{tor}}\] can be naturally
lifted to a subset $\Phi(H^{I,0},T^{I,0})\subset\xch(T)_I$ (by
considering the action of $T^I$ on $\Lie (H^{I,0})$). The exact
sequence
\[1\to \on{Int}((H^\vee)^I)\to \Aut((H^\vee)^I)\to \on{Out}((H^\vee)^I)\to 1,\]
is split with a splitting given by
$\Aut((H^\vee)^I,(B^\vee)^I,(T^\vee)^I,X)\simeq
\on{Out}((H^\vee)^I)$ In addition, we can identify
$\on{Out}((H^\vee)^I)$ with the group of automorphisms of the root
datum. Another example is provided in Proposition \ref{parallel}.
 }\end{rmk}
}
\medskip

Now, we switch to Theorem \ref{main finite}. Therefore, we will
assume that $G$ is a quasi-split reductive group defined over the
non-archimedean local field $F=\bbF_q\ppart$ (we can in fact replace
$\bbF_q$ by any other perfect field). Let $k=\overline{\bbF}_q$, and
$\sigma$ be the Frobenius element in $\Gal(k/\bbF_q)$. Let
$v\in\calB(G,F)$ be a special vertex in the building which remains
to be special when base change to $k\ppart$ (such a vertex is called
very special ,see \S \ref{parameter} for more discussions). Let
$\underline{G}_v$ be the special parahoric group scheme over
$\bbF_q[[t]]$ corresponding to $v$ and $K_v=L^+\underline{G}_v$.
Then the affine flag variety $\Fl_v=LG/K_v$ is defined over $\bbF_q$
and when base change to $k$, $\Fl_v\otimes k$ is the affine flag
variety considered in the previous sections, and we have the
Tannakian category $\calP_v=\calP_{K_v\otimes k}(\Fl_v\otimes k)$
with a fiber functor $\rH^*$.

As in Lemma \ref{action of Galois}, there is an action of $\sigma$
on $\calP_v$, and therefore an action of $\sigma$ on $(H^\vee)^I$.
Following the notation as in the appendix, we denote this action by
$\on{act}^{geom}$. On the other hand, since there is a canonical
pinning $(H^\vee,B_H^\vee,T_H^\vee,X^\vee)$, there is a canonical
action of $\Gal(F^s/F)$ on $H^\vee$ by pinned automorphisms and
therefore an action of $\sigma$ on $(H^\vee)^I$ by pinned
automorphisms. We denote this action by $\on{act}^{alg}$. As in the
Appendix, we denote $\on{cycl}$ to be the cyclotomic character of
$\Gal(k/\bbF_q)$, so that $\on{cycl}(\sigma)=q$. Let
\[\chi=\rho\circ\on{cycl}:\Gal(k/\bbF_q)\to (H^\vee_\ad)^I.\]  As in Proposition \ref{comparison},
since the action of $\on{act}^{geom}$ on $(H^\vee)^I$ fixes the
cohomological grading and acts on $X^\vee$ via the cyclotomic
character, we know that
\[\on{act}^{geom}=\on{act}^{alg}\circ\Ad_\chi,\]
and there is an isomorphism
\begin{equation}\label{geom-to-alg}(H^\vee)^I\rtimes_{\on{act}^{alg}}\Gal(k/\bbF_q)\to
(H^\vee)^I\rtimes_{\on{act}^{geom}}\Gal(k/\bbF_q), \quad
(g,\sigma)\mapsto (\Ad_{\chi(\sigma)^{-1}}g,\sigma).\end{equation}
Now regarding $(H^\vee)^I\rtimes_{\on{act}^{alg}}\Gal(k/\bbF_q)$ as
a pro-algebraic group over $\overline{\bbQ}_\ell$, as in the
Appendix, we have the category
$\Rep((H^\vee)^I\rtimes_{\on{act}^{alg}}\Gal(k/\bbF_q))$ of
algebraic representations of
$(H^\vee)^I\rtimes_{\on{act}^{alg}}\Gal(k/\bbF_q)$. Now Theorem
\ref{main finite} follows from the same line as in the Appendix.

\section{$\IC$-stalks, $q$-analogy of the weight multiplicity, and the Lusztig-Kato
polynomial}\label{stalk} Let $\bar{\mu}\in\xcoch(T)_I$ and
$\calF\in\calP_v$. We determine the stalk cohomology $\calF$ at the
point $s_{\bar{\mu}}$. By abuse of notation, the inclusion map
$s_{\bar{\mu}}\in {\Fl_v}$ is still denoted by $s_{\bar{\mu}}$. It
will be convenient to define
\[\on{Stalk}_{\bar{\mu}}(\calF)=s_{\bar{\mu}}^*\calF[-(2\rho,\bar{\mu})],\quad \on{Costalk}_{\bar{\mu}}(\calF)=s_{\bar{\mu}}^!\calF[(2\rho,\bar{\mu})].\]
Let $X^\vee$ be the regular nilpotent element of $\Lie(H^\vee)^I$
given by the pinning. It defines an increasing filtration (the
Brylinski-Kostant filtration) on any representation $V$ of $H^\vee$
or $(H^\vee)^I$,
\begin{equation}\label{BK fil}
F_iV=(\ker X^\vee)^{i+1}.
\end{equation}
For $\bar{\mu}\in\xcoch(T)_I$, denote by $V(\bar{\mu})$ the
$\bar{\mu}$-weight subspace of $V$, under the action of
$(T_H^\vee)^I$. Then filtration \eqref{BK fil} induces
\begin{equation}\label{BKfilmu}
F_iV(\bar{\mu})=V(\bar{\mu})\cap F_iV.
\end{equation}
Let
\[P_{\bar{\mu}}(V,q)=\sum \on{gr}^F_iV(\bar{\mu})q^i\]
be the $q$-analogue weight multiplicity polynomial.

\begin{thm}\label{stalk cohomology}
Let $\calF\in\calP_v$ and let $V=\rH^*({\Fl_v},\calF)$ be the
corresponding representation of $(H^\vee)^I$. Then
\[P_{\bar{\mu}}(V,q)=\sum \dim \rH^{-2i}(\on{Stalk}_{\bar{\mu}}(\calF))q^{i}=\sum \dim \rH^{2i}(\on{Costalk}_{\bar{\mu}}(\calF))q^{i}.\]
\end{thm}

Observe that by the parity vanishing property of $\calF$,
$\on{Stalk}_{\bar{\mu}}\calF$ and $\on{Costalk}_{\bar{\mu}}(\calF)$
only concentrate on even degrees.

In the split case, this is proved in \cite{Lu,Bry}. A more geometric
proof is given by Ginzburg \cite{Gi}, which relies on the geometric
Satake isomorphism and certain techniques of equivariant cohomology.
We will follow Ginzburg's idea.

\medskip

Let us give a quick review of equivariant cohomology (see \cite{BL}
and \cite[\S 8]{Gi} for more details). Let $M$ be a variety with an
action of a torus $A$. Let $\bbB A$ be the classifying space (stack)
of $A$. Let $R_A=\rH^*(\bbB A)$ and recall that $\Spec R_A\simeq
\fraka=\Lie A$. Let $t\in\fraka$ be an element. We denote
$\kappa(t)$ be the residue field of $t$ and let
$\rH_t:=\rH^*_A\otimes_{R_A}\kappa(t)$. If $t=0$, this functor
$\rH_0$ inherits a canonical grading. For general $t$, this functor
equips with a canonical filtration by
\[\rH_t^{\leq i}:=\on{Im}(\sum_{j\leq i}\rH_A^j\to \rH_t),\]
and there is a canonical isomorphism $\on{gr}^*\rH_t\simeq \rH^*_0$.
For every $\calF\in D_A(M)$, there is a spectral sequence
$E_2^{p,q}=\rH^p(\bbB A,H^q(M,\calF))\Rightarrow
\rH^{p+q}_A(M,\calF)$. If this spectral sequence degenerates at the
$E_2$-term (which is always the case in the following discussion),
then $\rH_0^*(M,\calF)\simeq \rH^*(M,\calF)$ and therefore we have a
canonical isomorphism $\on{gr}\rH_t(M,\calF)\simeq \rH^*(M,\calF)$.

Now assume that the action of $A$ on $M$ has only isolated fixed
points $M^A$, and let $\eta\in\fraka$ be the generic point. Then the
localization theorem claims that there is an isomorphism
\begin{equation}\label{localization} \bigoplus_{x\in
M^A}\rH_\eta(i_x^!\calF)\simeq\rH_\eta(M,\calF)\simeq
\bigoplus_{x\in M^A}\rH_\eta(i_x^*\calF),\end{equation} where $i_x$
is the inclusion of the point $x$.

\medskip

Now consider $\rH^*_{T_H}:\Sat_H\to R_{T_H}\Mod$. It is proved in
\cite[Lemma 2.2]{YZ} that there is a canonical grading preserving
isomorphism
\begin{equation}\label{triv}
\rH_{T_H}^*\simeq \rH^*\otimes R_{T_H}:\Sat_H\to R_{T_H}\Mod,
\end{equation}
which endows $\rH_{T_H}^*$ with a structure of tensor functors and
defines a canonically trivialized $H^\vee$-torsor $\calE\simeq
H^\vee\times\frakt_H$ on $\Spec R_{T_H}=\frakt_H=:\Lie T_H$. In
other words, the group scheme $\Aut^{\otimes} \rH^*_{T_H}$ over
$\frakt_H$ of the tensor automorphism of this fiber functor, which a
priori is an inner form of $H^\vee$, is canonically isomorphic to
$H^\vee\times\frakt_H$. In addition, the MV filtration and its
canonical splitting extend in the equivariant setting \cite[Lemma
2.2]{YZ} and provide $T_H^\vee\times\frakt_H\subset
B_H^\vee\times\frakt_H\subset H^\vee\times\frakt_H$. Now, let
$c^{T_H}(\calL)\in \rH^2_{T_H}(\Gr_H)$ denote the equivariant Chern
class of $\calL$\footnote{by replacing $\calL$ be a power of it, we
can assume that $\calL$ is $T_H$-equivariant.}. Then the action of
$c^{T_H}(\calL)$ on $\rH_{T_H}^*(\Gr_H,\calF)$ for $\calF\in\Sat_H$
can be identified with the action of an element
\[e^{T_H}\in \Gamma(\frakt_H,\Lie(\ad\calE)).\]
Since $\calE$ is canonically trivialized, $e^{T_H}$ can be regarded
as a map $\frakt_H\to\frakh^\vee$. Observe that $e^{T_H}$ is NOT the
constant map $X^\vee$. In fact,
\[e^{T_H}=X^\vee+h,\]
where $h:\frakt_{H}\to\frakt^\vee_H\simeq (\frakt_H)^*$ is given by
a nondegenerate invariant bilinear form (cf. \cite[Proposition
5.7]{YZ}). In particular, $(H^\vee,B^\vee_{H}, T^\vee_{H}, e^{T_H})$
is not a pinning over $\frakt_H$.

The equivariant homology $\rH^{T_H}_*(\Gr_H)$ is a commutative and
cocommutative Hopf algebra and $J^\vee=\Spec \rH^{T_H}_*(\Gr_H)$ is
a flat group scheme over $\frakt_H$, acting on every
$\rH_{T_H}^*(\Gr_H,\calF), \calF\in\Sat_H$. By Tannakian formulism,
this induces a map $\iota: J^\vee\to H^\vee\times \frakt_H$. In
\cite{YZ}, it is shown that this is a closed embedding, which
identifies $J^\vee$ with the $(H^\vee\times\frakt_H)^{e^{T_H}}$, the
centralizer of $e^{T_H}$ in $H^\vee\times\frakt_H$.

Let $\eta$ be the generic point of $\frakt_H$. Then $J^\vee_\eta$ is
indeed a torus in $H^\vee_\eta$ since $e^{T_H}(\eta)\in
\frakh^\vee_\eta$ is regular semisimple. Then localization theorem
gives rise to an isomorphism of $J^\vee_\eta$-modules
\begin{equation}\label{loc}
\bigoplus_{\mu\in\xcoch(T_H)}\rH_\eta(s_\mu^!\calF)\simeq\rH_\eta(\Gr_H,\calF).
\end{equation}
Following the idea of Ginzburg, we claim that this decomposition
corresponds to the weight decomposition under $J^\vee_\eta\subset
H^\vee_\eta$. First, let $B_H^\vee\to T_H^\vee$ be the natural
projection. As is shown in \cite{YZ}, $J^\vee\subset
B^\vee_H\times\frakt_H$ and the composition $J^\vee\to
B^\vee_H\times\frakt_H\to T^\vee_H\times\frakt_H$ is identified with
the map (cf. Remark 3.4 of \emph{loc. cit.})
\[R_{T_H}[\xcoch(T_H)]\simeq\bigoplus_{\mu\in\xcoch(T_H)} \rH^{T_H}_*(s_\mu)\to \rH^{T_H}_*(\Gr_H).\]
Over $\eta$, this is an isomorphism and therefore we obtain a
canonical isomorphism $J^\vee_\eta\simeq (T^\vee_H)_\eta$. In
addition, the action of $J^\vee_\eta$ on $\rH_{\eta,
c}(S_\mu,\calF)$ via $J^\vee_\eta\to (T_H^\vee)_\eta$ is identified
with the natural action of $J^\vee_\eta$ on
$\rH_\eta(s_\mu^!\calF)\simeq \rH_{\eta,c}(S_\mu,\calF)$. Therefore,
we obtain the following proposition, originally proved by Ginzburg
by another method.

\begin{prop}\label{wt decomposition}
Let $V\in \Rep(H^\vee)$ and $\calS(V)\in\Sat_H$ be the corresponding
sheaf (see \eqref{sat}). Under the identification of the weight
lattice of $J^\vee_\eta$ with $\xcoch(T_H^\vee)$ via the canonical
isomorphism $J^\vee_\eta\to (T^\vee_H)_\eta$, the direct summand
$\rH_\eta(s_\mu^!\calS(V))\subset \rH_\eta(\Gr_H,\calS(V))$
corresponds to the weight subspace $V(\mu)\subset V$ for
$J^\vee_\eta$.
\end{prop}

\begin{rmk}{\rm Let us observe that the localization isomorphism \eqref{loc}
holds over $\frakt_H$ after we remove all root hyperplanes. This is
because for every $T_H$-invariant finite dimensional closed
subvariety $Z\subset \Gr_H$, there are only finitely many
1-dimensional $T_H$-orbits in $Z$, and $T_H$ acts on these orbits
via rotations determined by roots. Therefore, in all the discussions
above, we can replace the generic point $\eta$ by any (closed) point
in $\frakt_H\setminus\{\tilde{a}=0, \tilde{a}\in\Phi_H\}$, where
$\Phi_H$ is the set of roots of $H$.
 }\end{rmk}

Now let $t$ be a closed point on $\frakt_H$ such that
\begin{equation}\label{pt t}
h(t)=2\rho
\end{equation}
so that $e^{T_H}(t)=X^\vee+h(t)=X^\vee+2\rho$. According to
\cite[Proposition 5.7]{YZ}, such point exists (unique up to adding
an element in the center $\frakz(\frakh)$ of $\frakh$) and does not
belong to any root hyperplanes. Therefore, the localization
isomorphism \eqref{loc} holds for $\rH_t$ by the above remark. From
now on, we will always choose the point $t$ satisfying \eqref{pt t}.

Recall that under the geometric Satake isomorphism
$\rH^*:\Sat_H\simeq \Rep(H^\vee)$, the natural grading on the
cohomology functor corresponds to the principal grading on
representations of $H^\vee$. More precisely, consider the
cocharacter $2\rho:\bbG_m\to T_H^\vee\subset H^\vee$. Then the
grading on the cohomology functor corresponds to the grading given
by $2\rho$ on the representations. This follows from the fact that
$\rH_c^*(S_\mu,\calF)$ is nonzero only in degree $(2\rho,\mu)$. Now
it is clear from \eqref{triv} that for the closed point
$t\in\frakt_H$, the filtration $\rH_t^{\leq i}$ corresponds to the
increasing filtration on the representations associated to the
gradings given by $2\rho$. For $i\in\bbZ$, let
\[\xcoch(T_H)_i=\{\mu\in\xcoch(T_H)\mid (2\rho,\mu)=i\}.\]
Let $V$ be a representation of $H^\vee$. Denote
\[V(i)=\sum_{\mu\in\xcoch(T_H)_i}V(\mu),\]
where $V(\mu)$ is the $\mu$-weight space of $J^\vee_t$. Let us
identify $\rH_t(\Gr_H,\calS(V))$ with $V$ canonically, so that
$\rH_t(s_\mu^*\calS(V))$ is identified with $V(\mu)$ by Proposition
\ref{wt decomposition}). Then we have the following proposition.

\begin{prop}\label{two fil}
Let $t$ be as in \eqref{pt t}. Write $\rH_t=\rH_t(\Gr_H,\calS(V))$
for simplicity. Then for any $m\in\bbZ$,
\[\rH_t^{\leq 2i+m}\cap \bigoplus_{\mu\in\xcoch(T)_m}\rH_t(s_\mu^!\calS(V))=\rH_t^{\leq 2i+1+m}\cap \bigoplus_{\mu\in\xcoch(T)_m}\rH_t(s_\mu^!\calS(V))= F_iV\cap V(m),\]
where $F_iV$ is defined as in \eqref{BK fil}.
\end{prop}
\begin{proof}
Let $n$ be the unique element in $U^\vee_H$ such that
$\Ad_n(X^\vee+2\rho)=2\rho$. Then the canonical isomorphism
$J^\vee_t\simeq T_H^\vee$ is given by $\Ad_n: J^\vee_t\to T_H^\vee$.
The proposition clearly follows from the following purely
representation theoretical lemma.
\end{proof}

\begin{lem}Let $V$ be a representation of $H^\vee$. Let
\[V=\sum V^1(i), \quad V=\sum V^2(i)\]
be two gradings on $V$, given by the cocharacters $2\rho:\bbG_m\to
H^\vee$ and $\Ad_{n^{-1}}2\rho:\bbG_m\to H^\vee$ respectively. Let
$F^1_\bullet V$ and $F^2_\bullet V$ be two filtrations on $V$ given
by
\[F^1_iV= \sum_{j\leq i} V^1(j),\quad F^2_iV=(\ker X^\vee)^{i+1}.\]
The for any $m\in\bbZ$, $V^2(m)\cap F^1_{2i+m}V=V^2(m)\cap
F^1_{2i+1+m}V= V^2(m)\cap F^2_iV$.
\end{lem}
\begin{proof}Let $Y^\vee\in\frakh^\vee$ so that
$\{X^\vee,2\rho,Y^\vee\}$ form an $\fraks\frakl_2$-triple. Then the
lemma is purely a statement about this $\fraks\frakl_2$ and can be
checked easily by direct calculation.
\end{proof}

This finishes the discussion for split groups. Now let $G$ be as
before and $A$ be its maximal split torus. By the isomorphism
\eqref{rho}, we can regard $A$ as a subtorus of $T_H$. Note that we
can restrict everything discussed above to $A\subset T_H$. In
particular, we can choose the point $t\in\fraka$ such that
$h(t)=2\rho$. This is because that $h:\frakt_H\to\frakt^\vee_H$ is
equivariant under the automorphisms of the based root datum and
$2\rho$ is a fixed point under these automorphisms.

We begin to prove the theorem. Observe that it is enough to prove
the theorem for objects in $\calP_v$ of the form $\calZ(\calF)$,
where $\calF\in\Sat_H$. Indeed, both maps $\calP_v\to \bbZ[q]$ given
by $\calF\mapsto P_{\bar{\mu}}(\rH^*(\calF),q)$ and  $\calF\mapsto
\sum \dim \rH^{2i}(\on{Costalk}_{\bar{\mu}}(\calF))q^{i}$ factor
through the Grothendieck group, and as observed in the proof of
Theorem \ref{main} in Sect. \ref{proof}, the objects of the form
$\calZ(\calF), \calF\in\Sat_H$ generate the Grothendieck group of
$\calP_v$.

Recall that the maximal torus $A\subset G$ extends naturally to a
split torus over $\calO$ and $A_\calO\subset\underline{G}_v$.
Therefore, we can regard $A$ as a subtorus of $K_v$, as $A$ is a
natural subgroup of $L^+A_\calO$ consisting of ``constant" elements.
The set of fixed points of the action of $A$ on ${\Fl_v}$ are
exactly $\{s_{\bar{\mu}}|\bar{\mu}\in\xcoch(T)_I\}$. This will be
clear if we regard $LG$ as a Kac-Moody group and $A$ as its maximal
torus. We consider the $A$-equivariant cohomology
$\rH_A^*:\calP_v\to R_A\Mod$. Since the nearby cycles commute with
proper base change, we have a canonical isomorphism
\[\rH_A^*\circ\calZ\simeq \rH_A^*: \Sat_H\to R_A\Mod.\]
Indeed, when fixing a cohomological degree, we can we replace $\bbB
A$ by $(\bbP^n)^{\on{rk}A}$ for $n$ large enough, and consider the
nearby cycle functors for the family
$\wGr_\calG\times^A(\bbP^n)^{\on{rk}A}$. The claim then is clear.

\begin{rmk}{\rm
One should be able to argue as defining the tensor structure of
$\rH^*$, that there is a canonical isomorphism
\[\rH^*_A\simeq \rH^*\otimes R_A:\calP_v\to R_A\Mod,\]
which endows $\rH^*_A$ a fiber functor structure, and the
corresponding $(H^\vee)^I$ torsor on $\Spec R_A=\fraka=:\Lie A$ is
canonically trivialized. However, we did not investigate this.
 }\end{rmk}

Let $p:\xcoch(T)\to \xcoch(T)_I$ the projection. Let $\eta$ be the
generic point of $\fraka$.
\begin{lem} Under the canonical
isomorphism $\rH_\eta({\Fl_v},\calZ(\calF))\simeq
\rH_\eta(\Gr_H,\calF)$, the direct summand
$\rH_\eta^*(s^!_{\bar{\mu}}\calZ(\calF))$ corresponds to
$\bigoplus_{\mu\in p^{-1}(\bar{\mu})}\rH_\eta^*(s_\mu^!\calF)$.
Therefore, by Proposition \ref{wt decomposition}, if
$\calF=\calS(V)$, this direct summand can be further identified with
$\bigoplus_{\mu\in p^{-1}(\bar{\mu})}V(\mu)$, the weight subspaces
under $J^\vee_\eta$.
\end{lem}
\begin{proof}
Recall that if $f:\calX\to\calY$ is a morphism of varieties over
$\bbA^1$, then there are always the natural maps $f^*\Psi_\calY\to
\Psi_\calX f^*$ and $\Psi_\calY f_*\to f_*\Psi_\calX$ (see SGA 7
Expos\'{e} XIII, (2.1.7.1) (2.1.7.2)). In addition, these two maps
fit into the following commutative diagram
\[\begin{CD}
\Psi_\calY(\calF)@>>>f_*f^*\Psi_\calY(\calF)\\
@VVV@VVV\\
\Psi_\calY(f_*f^*\calF)@>>>f_*\Psi_\calX(f^*\calF).
\end{CD}\]

Now let $\mu\in\xcoch(T_H)$ and apply this remark to
$s_\mu:\bbA^1\to\wGr_\calG$ as defined in the proof of Lemma
\ref{mult}. By taking the cohomology $\rH_A^*$, we obtain, for any
$\calF\in\Sat_H$, the following commutative diagram
\begin{equation}\label{comm1}\begin{CD}
\rH^*_A(\Gr_H,\calF)@>\simeq>>\rH_A^*({\Fl_v},\calZ(\calF))@>>>\rH_A^*(s_{\bar{\mu}}^*\calZ(\calF))\\
@VVV@VVV@VVV \\
\rH_A^*(s^*_\mu\calF)@>\simeq>>\rH_A^*({\Fl_v},\Psi_{\wGr_\calG}(s_{\mu*}s_\mu^*\calF))@>\simeq>>
\rH_A^*(\Psi_{\bbA^1}(s^*_\mu\calF))
\end{CD}\end{equation}
In other words, the composition $\rH_A^*({\Fl_v},\calZ(\calF))\to
\rH_A^*(\Gr_H,\calF)\to \rH_A^*(s_\mu^*\calF)$ factors as
$\rH_A^*({\Fl_v},\calZ(\calF))\to
\rH_A^*(s_{\bar{\mu}}^*\calZ(\calF))\to \rH_A^*(s_\mu^*\calF)$. On
the other hand, by the localization theorem, over the generic $\eta$
of $\fraka$, we have
\begin{equation}\label{comm2}\begin{CD}
\rH_\eta^*({\Fl_v},\calZ(\calF)))@>\simeq>>\bigoplus_{\bar{\mu}\in\xcoch(T)_I}\rH_\eta^*(s_{\bar{\mu}}^*\calZ(\calF))\\
@V\simeq VV@VV\simeq V\\
\rH_\eta^*(\Gr_H,\calF)@>\simeq>>\bigoplus_{\mu\in\xcoch(T)}\rH_\eta^*(s_{\mu}^*\calF)\end{CD}\end{equation}
Observe that in the localization theorem \eqref{localization}, for
$x,y\in M^A$ and $x\neq y$, the composition
$\rH_{\eta}(i_x^!\calF)\to \rH_{\eta}(i_y^*\calF)$ is zero.
Therefore, the lemma follows from \eqref{comm1} and \eqref{comm2}.
\end{proof}

Now by Proposition \ref{two fil},
\[\rH_t^{\leq 2i+(2\rho,\bar{\mu}))}\cap \rH_t(s_{\bar{\mu}}^!\calZ(\calS(V)))=\rH_t^{\leq 2i+1+(2\rho,\bar{\mu}))}\cap \rH_t(s_{\bar{\mu}}^!\calZ(\calS(V)))=F_iV\cap \bigoplus_{\mu\in p^{-1}(\bar{\mu})}V(\mu),\]
where we write $\rH_t=\rH_t({\Fl_v},\calZ(\calS(V)))$ for brevity.
Therefore, to finish the prove of the theorem, it remains to show
that

\begin{lem}\label{splitting inj}
Let $\calF\in\calP_v$. Then the canonical map
\begin{equation}\label{inj}\rH^*_A(s_{\bar{\mu}}^!\calF)\to\rH_A^*(\calF)\end{equation} is a splitting
injective map of free $R_A$-modules. Therefore,
\[\rH_t^{\leq i}({\Fl_v},\calZ(\calS(V)))\cap \rH_t(s_{\bar{\mu}}^!\calZ(\calS(V)))=\rH_t^{\leq i}(s_{\bar{\mu}}^!\calZ(\calS(V))).\]
\end{lem}

These are general facts about flag varieties for Kac-Moody groups.
The basic geometric fact behind this proposition is that the ``big
open cell" of the flag variety contracts to a point under certain
$\bbG_m$-action. Then the statement follows using an argument with
weights (cf. \cite{Gi}). We here reproduce the proof for
completeness.

\begin{proof}Without loss of generality, we can assume that $\calF$
is an intersection cohomology complex. First, we claim that it is
enough to prove a dual statement: the map
\begin{equation}\label{surj}\rH_A^*(\calF)\to \rH_A^*(s_{\bar{\mu}}^*\calF)\end{equation} is
surjective. To see this, recall that since each of $\rH^*(\calF),
s_{\bar{\mu}}^*\calF,s_{\bar{\mu}}^!\calF$ concentrates in
cohomological degrees of the same parity, the spectral sequence
calculating the $A$-equivariant cohomology degenerates at the
$E_2$-term, which implies that all $\rH_A^*(\calF)$,
$\rH_A^*(s_{\bar{\mu}}^*\calF)$, $\rH_A^*(s_{\bar{\mu}}^!\calF)$ are
finite free $R_A$-modules. Then taking $\Hom(-,R_A)$ interchanges
\eqref{inj} and \eqref{surj}.

Since ${\Fl_v}$ is the flag variety of certain Kac-Moody group (cf.
\cite[Sect. 9.h]{PR}), for every $s_{\bar{\mu}}$, there is a
$\bbG_m$-action on ${\Fl_v}$, contracting an open neighborhood of
$s_{\bar{\mu}}$ in ${\Fl_v}$ to $s_{\bar{\mu}}$. In addition, this
$\bbG_m$-action stabilizes every Schubert cell
$\mathring{\Fl}^s_{\bar{\la}}$, and commutes with the action of $A$
on ${\Fl_v}$. Denote this open neighborhood by
$j:U_{\bar{\mu}}\hookrightarrow{\Fl_v}$. Then $U_{\bar{\mu}}$ is an
inductive limit of affine spaces. Indeed, $U_{\bar{0}}$ is just the
big open cell in the flag variety, and $U_{\bar{\mu}}$ is the
translate $U_{\bar{0}}$ via $s_{\bar{\mu}}$ (lifted to an element in
$T(F)$).

Now we can assume that our group is defined over $\bbF_q\ppart$ and
splits over a totally ramified extension. All the discussion above
remains unchanged in this setting. Recall that we denote $\calP^0_v$
to be the semisimple $K_v$-equivariant perverse sheaves on
${\Fl_v}$, pure of weight zero. It is well-known that for every
object $\calF$ in $\calP_v$, $\rH^*_A(s_{\bar{\mu}}^*\calF)$ is pure
of weight zero (i.e. $\rH^i_A(s_{\bar{\mu}}^!\calF)$ is pure of
weight $i$), essentially due to the existence of Demazure
resolutions. To show \eqref{surj} is surjective, We decompose this
map into
\[\rH_A^*(\calF)\to\rH^*_{A}(U_{\bar{\mu}},j^*\calF)\to\rH^*_A(s_{\bar{\mu}}^*\calF).\]
It is well-known that the second map is an isomorphism since
$j^*\calF$ is equivariant under this $\bbG_m$-action, which
contracts $(U_{\bar{\mu}}\cap \on{Supp}\calF)$ to $s_{\bar{\mu}}$.
In particular, $\rH^*_{A}(U_{\bar{\mu}},j^*\calF)$ is pure of weight
zero. Therefore, it is enough to show that the first map is
surjective. Denote $i:Z={\Fl_v}\setminus
U_{\bar{\mu}}\hookrightarrow {\Fl_v}$ to be the complement. Then we
have the distinguished triangle
\[i_*i^!\calF\to\calF\to j_*j^*\calF\to\]
and therefore
\[0\to\rH_{A}^*(i^!\calF)\to\rH^*_A(\calF)\to \rH^*_{A}(U_{\bar{\mu}},j^*\calF)\to 0.\]
The last map is surjective because the weights of
$\rH^*_{A}(i^!\calF)$ are $\geq 0$.
\end{proof}

\section{The Langlands parameter}\label{parameter}
In this section, we briefly discuss the Langlands parameters for
smooth ``spherical" representations of a quasi-split $p$-adic group.
The parameters themselves can be described easily, and they will be
used when we discuss the Frobenius trace of nearby cycles for
certain unitary Shimura varieties.

We will assume that $F$ is a non-archimedean local field with finite
residue field and that $G$ is a connected reductive group over $F$.
First, we generalize the hyperspecial vertex of an uniramfied group
as follows. Recall by \cite{T}, the building of $G(F)$ can be
embedded into the building of $G(L)$, where $L$ is the completion of
a maximal unramified extension of $F$.

\begin{dfn}\label{very special}
A special vertex $v$ of $G$ is called geometrically special (or very
special) if it remains special in $G_L$. The parahoric subgroup of
$G$ corresponding to a geometrically special vertex is called a
geometrically special (or very special) parahoric subgroup of $G$.
\end{dfn}

Clearly, if $G$ is an unramified group, then very special vertices
of $G$ are the same as hyperspecial vertices of $G$.

\begin{lem}A very special vertex of $G$ exists if and only if $G$ is quasi-split over $F$.
\end{lem}
\begin{proof}Assume that $G$ is quasi-split. Then the existence of such points follows exactly by the same
argument as in \cite[1.10.2]{T}. We prove the converse. Let $v$ be a
very special point. Choose a maximal $F$-split torus $A$ of $G$ such
that the corresponding apartment $A(G,A,F)$ containing $v$. Let $S$
be a maximal $L$-split torus defined over $F$ and containing $A$. We
identify the apartment $A(G_L,S_L,L)$ with $\xcoch(S)\otimes\bbR$ by
$v$. As $v$ is special, there is a bijection between the finite Weyl
chambers for $(G_L,S_L)$ and the affine Weyl chambers (or called
alcove) with $v$ as a vertex, and this bijection is compatible with
the action of $\Gal(L/F)$. To show that $G$ is quasi-split, it is
enough to find a $L$-rational Borel containing $S$ stable under
$\Gal(L/F)$, which is equivalent to finding a finite Weyl chamber in
$\xcoch(S)\otimes\bbR$, stable under $\Gal(L/F)$. Therefore, it
enough to show that among all alcoves with $v$ as a vertex, there is
one stable under $\Gal(L/F)$. But as it is known, one of such
alcoves intersects with $A(G,A,F)$ (since every reductive group over
$F$ is residually quasi-split, see \cite[\S 1.10]{T}), which is
stable under $\Gal(L/F)$.
\end{proof}

In fact, by checking the classification of central isogeny classes
of quasi-simple, absolutely simple reductive group over $F$ as in
\cite[\S 4]{T}, we find that if $G$ is quasi-split, then every
special vertex of $G$ is very special except the following case: up
to central isogeny, $G$ is an unramified odd unitary group. Then
there are two special vertices in its relative local Dynkin diagram,
only one of which is hyperspecial. To prove this assertion, one uses
the following observation: Using the notation as in \emph{loc.
cit.}, a vertex $v$ in $\Delta$ (the relative local Dynkin diagram
of $G$) is very special if and only if the corresponding
$\Gal(L/F)$-orbit $O(v)\subset \Delta_1$ (the absolute local Dynkin
diagram of $G$) consist of one point, which is special in
$\Delta_1$.

Next we turn to representations.
\begin{dfn}
We call an irreducible smooth representation $V$ of $G$ spherical,
if there is some $v\in V, v\neq 0$, which is fixed by some very
special parahoric subgroup of $G$.
\end{dfn}

\begin{rmk}{\rm Again, one could try to define spherical
representations of $G$ as those with a vector fixed by some special
maximal compact subgroup of $G$. However, from the point of view of
Langlands parameters discussed below, this is not correct.
 }\end{rmk}

Clearly, if $G$ is unramified, spherical representations are those
usually called unramified representations. For the unramified
representations, the description the associated Langlands parameters
is well-known (for example, see \cite[Chapter II]{Bo}). Let us
explain the Langlands parameters of spherical representations for
quasi-split ramified groups.

Following the notation in the previous sections, we denote by
$H^\vee$ its dual group in the sense of Langlands defined over
$\bbC$, i.e. the root datum is dual to the root datum of $G$. Let us
equip $H^\vee$ with a pinning
$(H^\vee,B^\vee,T^\vee,X^\vee)$\footnote{In fact, by the
construction of the Appendix, there is a canonical pinned of
$H^\vee$ provided by the geometric Satake correspondence.}. Then
$\Gal(F^s/F)$ acts on $H^\vee$ via pinned automorphisms, which we
denote by $\on{act}^{alg}$, and we can form the Langlands dual group
of $G$ as
\[{^L}G^{alg}=H^\vee\rtimes_{\on{act}^{alg}}\Gal(F^s/F).\]

Let $W_F$ be the Weil group of $F$. A Langlands parameter is a
continuous homomorphism (up to conjugation by $H^\vee$) $\rho:W_F\to
{^L}G^{alg}$, such that its composition with the canonical
projection ${^L}G^{alg}\to\Gal(F^s/F)$ is the natural inclusion
$W_F\to\Gal(F^s/F)$ and $\rho(W_F)$ consists of semisimple elements
of ${^L}G^{alg}$. (see \cite[8.2]{Bo} for the unexplained
terminology).

We write $\rho(\gamma)=(\rho_1(\ga),\ga)$ for $\ga\in W_F$, where
$\rho_1$ is a map from $W_F$ to $H^\vee$.

\begin{dfn}A ``spherical" parameter (or Langlands-Satake parameter) is a Langlands parameter $\rho$ which
can be conjugated to the form $\rho(\gamma)=(1,\ga)$ for $\ga$ in
the inertial group $I$.
\end{dfn}

Let $(H^\vee)^I$ be the $I$-fixed point subgroup of $H^\vee$ (which
could be non-connected according to Remark \ref{non-connected}).
Then $\Gal(F^{s}/F)/I$ acts on $(H^\vee)^I$ through a finite cyclic
group $\langle\sigma\rangle$, where $\sigma\in\Gal(F^{s}/F)/I$ is
the Frobenius element.

\begin{lem}
``Spherical" Langlands parameters $\rho:W_F\to {^L}G^{alg}$ are in
one-to-one correspondence to semi-simple elements in
$(H^\vee)^{I}\times
\sigma\subset(H^\vee)^I\rtimes_{\on{act}^{alg}}\langle\sigma\rangle$
up to conjugacy by $(H^\vee)^I$.
\end{lem}
We denote the set of semi-simple elements in $(H^\vee)^{I}\times
\sigma$ by $((H^\vee)^{I}\times \sigma)_{ss}$.
\begin{proof}
First, observe that elements $h\in H^\vee$ satisfies that
$(h,1)(1,\ga)(h^{-1},1)=(1,\ga)$ for all $\ga\in I$ if and only if
$h\in (H^\vee)^I$.

Let $\Phi$ be a lift of the Frobenius element to $W_F$. Then $\rho$
is uniquely determined by $\rho_1(\Phi)$, which is a semi-simple
element in $H^\vee$. Let $\ga\in I$. Then
$(1,\Phi\ga\Phi^{-1})=\rho(\Phi\ga\Phi^{-1})=(\rho_1(\Phi),\Phi)(1,\ga)(\rho_1(\Phi^{-1}),\Phi^{-1})$
implies that $\rho_1(\Phi)$ is invariant under $I$. Conversely, if
$g\in H^\vee\rtimes_{\on{act}^{alg}} I$, then the formulas
$\rho_1(I)=1, \rho_1(\Phi)=g$ define $\rho$.
\end{proof}

Let $A$ be a maximal $F$-split torus and let $T$ be the centralizer
of $A$, which is a maximal torus of $G$. Let $W_0=W(G,A)$ be the
Weyl group. As explained in \cite[Remark 9]{HRa}, we can identify
$W_0$ with the $\sigma$ invariants of the Weyl group
$W((H^\vee)^I,(T^\vee)^I)$ (observe that the latter group was
denoted by $W_0$ in Sect. \ref{proof}), and let $N_0$ be the inverse
image of $W_0$ inside the normalizer of $(T^\vee)^I$ in
$(H^\vee)^I$. Let
$\Rep((H^\vee)^I\rtimes_{\on{act}^{alg}}\langle\sigma\rangle)$ be
category of algebraic representations of
$(H^\vee)^I\rtimes_{\on{act}^{alg}}\langle\sigma\rangle$.

For every
$W\in\Rep((H^\vee)^I\rtimes_{\on{act}^{alg}}\langle\sigma\rangle)$,
by restriction of its character to $(H^\vee)^I\times\sigma$, we
obtain a function $\on{ch}_W$ on $(H^\vee)^I\times\sigma$. We denote
by $\on{R}$ the algebra of functions on $(H^\vee)^I\times\sigma$,
generated by all $\on{ch}_W$. We can adapt the proofs in
\cite[6.4-6.7]{Bo} to the group
$(H^\vee)^I\rtimes\langle\sigma\rangle$, and obtain
\begin{prop}\label{parallel}
(i) The natural map $\xcoch(T)_I^\sigma\subset\xcoch(T)_I$ induces
an isomorphism
\[\al:(T^\vee)^I\rtimes\sigma/\on{Int}N_0\simeq\Spec\bbC[\xcoch(T)_I^\sigma]^{W_0}.\]

(ii) The natural map $(T^\vee)^I\rtimes\sigma\to
(H^\vee)^I\rtimes\sigma$ induces an isomorphism
\[\beta:(T^\vee)^I\rtimes\sigma/\on{Int}N_0\simeq
((H^\vee)^I\rtimes\sigma)_{ss}/\on{Int}((H^\vee)^I).\]

(iii) The composition
$\beta\al^{-1}:\Spec\bbC[\xcoch(T)_I^\sigma]^{W_0}\to((H^\vee)^I\times\sigma)_{ss}/\on{Int}((H^\vee)^I)$
induces an isomorphism
\[\bbC[\xcoch(T)_I^\sigma]^{W_0}\simeq \on{R}\]
as functions on $((H^\vee)^I\times\sigma)_{ss}/\on{Int}(H^\vee)^I$.
\end{prop}

Therefore, the set of spherical Langlands parameters can be
identified with the set of all characters of $R$. Namely if $\rho$
is a spherical parameter, then the corresponding character
$\chi_\rho: R\to \bbC$ is given by
\begin{equation}\label{cha1}\chi_\rho(\on{ch}_W)=\tr(\rho(\Phi), W), \end{equation}
where we assume (after conjugation) that $\rho(\Phi)\in
((H^\vee)^I\times \sigma)_{ss}$, and $ \in
(H^\vee)^I\rtimes\langle\sigma\rangle$.

Now let us explain how to attach to a spherical representation its
spherical parameter. Let $\pi$ be a spherical representation of $G$,
such that $\pi^{K_v}\neq 0$ for a very special parahoric subgroup
$K_v\subset G(F)$. Therefore, $\pi$ determines a character
$\chi_\pi$ of $C_c(K_v\!\setminus\! G(F)/K_v)$ by
\begin{equation}\label{cha2}
\chi_\pi(f)=\tr(\pi(f)),
\end{equation}
where we fix a measure on $G(F)$ so that the volume of $K_v$ is one.

\begin{dfn}\label{LP}
We define the spherical parameter associated to $\pi$ to be the
unique Langlands parameter
\[\Sat(\pi):W_F\to{^L}G\]
such that $\chi_{\Sat(\pi)}=\chi_\pi$ under the Satake isomorphism
\begin{equation}\label{dict}C_c(K_v\!\setminus\! G(F)/K_v)\simeq
\on{R}.\end{equation}
\end{dfn}

As explained in Lemma \ref{classicalsatake},  in the case
$F=\bbF_q\ppart$, the isomorphism \eqref{dict} can be deduced from
\ref{main finite} under the sheaf-function dictionary. We come back
to the notation as in \S \ref{proof}, in particular
$k=\overline{\bbF}_q$. Let $\bar{\mu}\in\xcoch(T)_I$. By abuse of
notation, we denote the corresponding Schubert variety in
$\Fl_v\otimes k$ by $\Fl_{v\bar{\mu}}\otimes k$. If $\bar{\mu}$ is
defined $\bbF_q$, i.e. $\bar{\mu}\in(\xcoch(T)_I)^{\sigma}$, then
$\Fl_{v\bar{\mu}}\otimes k$ is also defined over $\bbF_q$ and we
denote the corresponding Schubert variety in $\Fl_v$ by
$\Fl_{v\bar{\mu}}$. In this case the intersection cohomology sheaf
$\IC_{\bar{\mu}}$ is naturally in $\calP_v^0$, and
$\rH^*(\IC_{\bar{\mu}})$ is a representation of
$(H^\vee)^I\rtimes_{\on{act}^{alg}}\Gal(k/\bbF_q)$. When restricted
to $(H^\vee)^I$, it is the highest representation $W_{\bar{\mu}}$.
By abuse of notation, this algebraic representation of
$(H^\vee)^I\rtimes_{\on{act}^{alg}}\Gal(k/\bbF_q)$ is still denoted
by $W_{\bar{\mu}}$. Let $A_{\bar{\mu}}\in C_c(K_v\!\setminus\!
G(F)/K_v)$ be the associated function under Grothendieck's
sheaf-function dictionary. Combining \eqref{cha1} and \eqref{cha2},
we have
\begin{equation}\label{trace}
\tr(\pi(A_\mu))=\tr(\Sat(\pi)(\Phi), W_{\bar{\mu}}).
\end{equation}

\section{Applications to the nearby cycles on certain Shimura
varieties}\label{Shimura} One of the main motivations of this work
is to calculate the nearby cycles for certain unitary Shimura
varieties. This is achieved by the so-called Rapoport-Zink-Pappas
local models.

Let $F/\bbQ$ be a quadratic imaginary field we fix an embedding
$F\subset\bbC$. Let $(W,\phi)$ be a hermitian space over $F/\bbQ$,
of dimension $n=\dim W\geq 3$. Let $G=\on{GU}(W,\phi)$ be the group
of unitary similitudes defined by
\[G(R)=\{g\in \GL_{F}(W\otimes_\bbQ R)\mid \phi(gv,gw)=c(g)\phi(v,w), c(g)\in R^\times\}.\]
Assume that $(W_\bbR,\phi_\bbR)\simeq (\bbC^n,H)$, where $H$ is the
standard Hermitian matrix on $\bbC^n$ of signature $(r,s)$, i.e.
$H=\on{diag}\{(-1)^{(s)},1^{(r)}\}$ is the diagonal matrix with $-1$
repeated at the first $s$ places and $1$ repeated at the remaining
$r$ places\footnote{The corresponding hermitian form is
$H(z,w)=\bar{\underline{z}}^tH\underline{w}$ for
$\underline{z},\underline{w}\in\bbC^n$.}. Without loss of
generality, we can assume that $s\leq r$. Let
$h:\Res_{\bbC/\bbR}\bbG_m\to G_{\bbR}$ be the homomorphism given by
$h(z)=\on{diag}\{z^{(s)},\bar{z}^{(r)}\}$. Let $K\subset G(\bbA_f)$
be an open compact subgroup, small enough (i.e. $K$ is contained in
some principal congruence subgroup for some $N\geq 3$). Then
associated to the data $(G,\{h\},K)$, one can define a Shimura
variety $Sh(G,K)$ over a number field $E$, where $E=\bbQ$ if $r=s$,
and $E=F$ if $r\neq s$. Let us recall that $h$ also determines a
conjugacy class of one parameter subgroups of $G_\bbC$ (the Shimura
cocharacter), defined over $E$. In our case,
$G_\bbC\simeq\GL_n\times\bbG_m$ and the one parameter subgroups are
conjugate to $\mu_{r,s}(z)=(\on{diag}\{z^{(s)},1^{(r)}\},z)$.

Let us fix a prime $p>2$ and assume that $F/\bbQ$ is ramified at
$p$. We denote $F_p$ (resp. $E_p$) the completion of $F$ (resp. $E$)
at the unique place over $p$. In addition, we assume that $(W,\phi)$
is a split hermitian form at $p$. In other words,
$(W,\phi)_{\bbQ_p}\simeq (F_p^n,J)$, where $J$ is the split
hermitian matrix on $F_p^n$ with all its anti-diagonal entries $1$,
and $0$ elsewhere. Observe that this assumption automatically holds
if $n$ is odd. Then $G_{\bbQ_p}$ is quasi-split. We will assume that
$K=K_pK^p\subset G(\bbQ_p)G(\bbA_f^p)$ and $K_p$ is a special
parahoric of $G_{\bbQ_p}$, which is automatically very special in
the sense of \eqref{very special}. Let us make this more concretely.

Let $\pi$ be a uniformizer of $F_p$ so that $\pi^2=ap$ with $a$ a
T\"{e}chimuller lifting of $\bbF_p^\times$. Let $\{e_1,\ldots,e_n\}$
be a basis of $F_p^n$ so that $\phi(e_i,e_j)$ is given by $J$. Let
\[\Lambda_i=\on{Span}_{\calO_{F_p}}\{\pi^{-1}e_1,\cdots,\pi^{-1}e_i,e_{i+1},\cdots,e_n\}.\]
If $n=2m+1$, we consider two integral models for $G_{\bbQ_p}$,
\begin{equation}\label{spodd}\underline{G}_{v_0}=\{g\in G, g\Lambda_0=\Lambda_0\}, \quad
\underline{G}_{v_1}=\{g\in G, g\Lambda_m=\Lambda_m\}.\end{equation}
If $n=2m$, we consider the integral model
\begin{equation}\label{speven}\underline{G}_v=\{g\in G, g\Lambda_m=\Lambda_m\}.\end{equation}
As explained in \cite[Sect. 1.2]{PR2}, these $\underline{G}_v$ are
special parahoric group schemes and essentially all special
parahoric group schemes of $G$ are conjugate to this ones.

Let $K_p=\underline{G}_{v}(\bbZ_p)$. In this case, the Shimura
variety $Sh(G,K)$ has a well-defined model over $\calO_{E_p}$, as in
\cite{PR2}. Let us denote the integral model by $Sh_{K_p}$. In
addition, there is the so-called local model diagram
\[\xymatrix{
&\widetilde{Sh}_{K_p}\ar_{\pi}[dl]\ar^{\varphi}[dr]&\\
Sh_{K_p}&&\on{M}^{\on{loc}}_{K_p} },\] where the scheme
$\on{M}^{\on{loc}}_{K_p}$, which is called the local model of
$Sh_{K_p}$, is projective over $\calO_{E_p}$ with an action of
$\underline{G}_v\otimes_{\bbZ_p}\calO_{E_p}$, and is \'{e}tale
locally isomorphic to $Sh_{K_p}$. In addition,
$\pi:\widetilde{Sh}_{K_p}\to Sh_{K_p}$ is a
$\underline{G}_{v}\otimes_{\bbZ_p}\calO_{E_p}$-torsor, and
$\varphi:\widetilde{Sh}_{K_p}\to \on{M}^{\on{loc}}_{K_p}$ is
$\underline{G}_v\otimes_{\bbZ_p}\calO_{E_p}$-equivariant and is
formally smooth. (cf. \cite{PR2} for details).

We are interested in the nearby cycle $\Psi_{Sh_{K_p}}(\bbQ_\ell)$,
which is an $\ell$-adic complex on
$Sh_{K_p}\otimes_{\calO_{E_p}}\overline{\bbF}_p$, on which
$\Ga=\Gal(\overline{\bbQ}_p/E)$ acts continuously, compatibly with
the action of $\Ga$ on
$\on{M}^{\on{loc}}_{K_p}\otimes\overline{\bbF}_p$ through
$\Ga\to\Gal(\overline{\bbF}_p/\bbF_p)$. From the local model
diagram, we have
\[\pi^*\Psi_{Sh_{K_p}}(\bbQ_\ell)\simeq \varphi^*\Psi_{\on{M}^{\on{loc}}_{K_p}}(\bbQ_\ell).\]
Therefore, it is essentially enough to determine
$\Psi_{\on{M}^{\on{loc}}_{K_p}}(\bbQ_\ell)$. For this purpose, we
need to recall the geometry of $\on{M}^{\on{loc}}_{K_p}$.

First, let $F'=\bbF_p\pparu$ be a ramified quadratic extension of
$\bbF_p\ppart$ with $u^2=at$, where $a\in\bbF_p^\times$ as before.
Let $W'=F'e_1+\cdots+F'e_n$ and $\phi'$ be a split hermitian form on
$W'$ given by $\phi'(e_i,e_{n+1-j})=\delta_{ij}$. Let $G'$ be the
corresponding unitary similitude group over $\bbF_p\ppart$. The
parahoric group scheme $\underline{G}_v$ of $G$ over $\bbZ_p$ has an
obvious counterpart $\underline{G}'_v$ over $\bbF_p[[t]]$. Namely,
consider
\[\Lambda'_i=\on{Span}_{\calO_{F'}}\{\pi^{-1}e_1,\cdots,\pi^{-1}e_i,e_{i+1},\cdots,e_n\}.\]
If $\underline{G}_v=G_{\bbQ_p}\cap\Aut(\Lambda_i)$, then
$\underline{G}'_{v}=G'\cap\Aut(\Lambda'_i)$. Observe that there is
an isomorphism
$\underline{G}_v\otimes\bbF_p\simeq\underline{G}'_v\otimes\bbF_p$
(given by the obvious identification of $(W,\phi)_{\bbF_p}\simeq
(W',\phi')_{\bbF_p}$). In addition, the Shimura cocharacter
$\mu_{r,s}$ makes sense as a cocharacter of
$G'\otimes(\bbF_p\ppart)^{s}$. Let ${\Fl_v}=LG'/L^+\underline{G}'_v$
be the associated affine flag variety considered before.

Now, we give the description of $\on{M}^{\on{loc}}_{K_p}$.  The
following statements can be extracted from \cite{PR2,Ri,PZ}.
\begin{prop}\label{geom}
(i) The generic fiber of $\on{M}^{\on{loc}}_{K_p}$ is isomorphic to
$\calP_{\mu_{r,s}}$, where $\calP_{\mu_{r,s}}$ is the variety of
maximal parabolic subgroups of $G_{E_p}$ of the type given by
$\mu_{r,s}$.

(ii) The special fiber is isomorphic to ${\Fl_v}_{\bar{\mu}_{r,s}}$
in an equivariant way. More precisely, the
$L^+\underline{G}'_v$-action on ${\Fl_v}_{\bar{\mu}_{r,s}}$ factors
through an action of $\underline{G}'_v\otimes\bbF_p$, and there is
an isomorphism
\[\on{M}^{\on{loc}}_{K_p}\simeq {\Fl_v}_{\bar{\mu}_{r,s}},\] intertwining
the $\underline{G}_v\otimes\bbF_p$ action on the left and this
$\underline{G}'_v\otimes\bbF_p$-action on the right.

(iii) The generic point of the special fiber
$\on{M}^{\on{loc}}_{K_p}\otimes\bbF_p$ is smooth in
$\on{M}^{\on{loc}}_{K_p}$.
\end{prop}

Having described the geometry of $\on{M}^{\on{loc}}_{K_p}$, let us
state the main theorem of this section. First, let $\calP_v$ be
either
\begin{enumerate}
\item the category of $L^+\underline{G}'_v$-equivariant Weil perverse
sheaves on ${\Fl_v}$, constant along each
$L^+\underline{G}'_v$-orbit; or
\item  the category of
$L^+\underline{G}'_v\otimes\overline{\bbF}_p$-equivariant perverse
sheaves on ${\Fl_v}\otimes\overline{\bbF}_p$.
\end{enumerate}
Likewise, we understand $\IC_{\bar{\mu}_{r,s}}$ either as a pure
perverse sheaf of weight zero, or just a geometric perverse sheaf.
By Theorem \ref{main}, we have:
\begin{enumerate}\item if $n=2m+1$
is odd,
\[\calR\calS: \Rep(\on{GO}_{2m+1})\simeq \calP_v;\]
\item if $n=2m$ is even,
\[\calR\calS: \Rep(\on{GSp}_{2m})\simeq \calP_v.\]
\end{enumerate}
Next, let $V$ be the standard representation of $\GL_n$ and
$V_{r,s}=\wedge^s V$ to be its $s$th wedge power. We extend
$V_{r,s}$ to a representation of $\GL_n\times\bbG_m$, on which
$\bbG_m$ acts via the homotheties.
\begin{thm}\label{Shi} Regard $V_{r,s}$ as a representation of $\on{GO}_{n}\subset\GL_n\times\bbG_m$ if
$n$ is odd, or of $\on{GSp}_n\subset\GL_n\times\bbG_m$ if $n$ is
even by restriction. Then:
\begin{enumerate}
\item Denote by $\Psi^{\on{geom}}$ the underlying complex of sheaves of $\Psi_{\on{M}^{\on{loc}}_{K_p}}[rs](\frac{rs}{2})$ on
$\on{M}^{\on{loc}}_{K_p}\otimes\overline{\bbF}_p$. Then
\[\Psi^{\on{geom}}\simeq \calR\calS(V_{r,s})\simeq\left\{\begin{array}{ll} \IC_{\bar{\mu}_{r,s}} & n \mbox{ odd } \\  \sum_{s'\geq 0, s-s'\in 2\bbZ_{\geq 0}}\IC_{\bar{\mu}_{n-s',s'}} & n \mbox{ even. }\end{array}\right.\]

\item Let $r\neq s$. Then the action of the inertial subgroup $I\subset\Ga$ on
$\Psi_{\on{M}^{\on{loc}}_{K_p}}$ is trivial so that
$\Psi_{\on{M}^{\on{loc}}_{K_p}}$ admits a structure as a Weil sheaf
on ${\Fl_v}$. In addition, as Weil sheaves,
\[\Psi_{\on{M}^{\on{loc}}_{K_p}}[rs](\frac{rs}{2})\simeq \calR\calS(V_{r,s}).\]

\item Let $r=s=m$ where $n=2m$. By (1) \[\Psi^{\on{geom}}\simeq \sum_{m'\geq 0,m-m'\in 2\bbZ_{\geq 0}}\IC_{\bar{\mu}_{n-m',m'}}.\]
The action of the inertial subgroup $I$ on
$\Psi_{\on{M}^{\on{loc}}_{K_p}}$ factors through
$I\to\Gal(F_p/\bbQ_p)\simeq \bbZ/2$. In addition, the action of
$\bbZ/2$ on $\IC_{\bar{\mu}_{n-m',m'}}$ is trivial if $4\mid m-m'$
and is through the non-trivial character if $4\nmid m-m'$. As Weil
sheaves,
\[(\Psi_{\on{M}^{\on{loc}}_{K_p}})^I[rs](\frac{rs}{2})\simeq \sum_{m'\geq 0,m-m'\in 4\bbZ_{\geq 0}}\IC_{\bar{\mu}_{n-m',m'}},\]
where $(\Psi_{\on{M}^{\on{loc}}_{K_p}})^I$ denotes the inertial
invariants of $\Psi_{\on{M}^{\on{loc}}_{K_p}}$.
\end{enumerate}
\end{thm}
\begin{proof}Observe that $\Psi^{\on{geom}}$ is an
object in $\calP_v$. This follows from Proposition \ref{geom} (ii).
Observe that $\IC_{\bar{\mu}_{r,s}}$ is a direct summand of
$\Psi^{\on{geom}}$. This follows from the fact that
$\on{M}^{\on{loc}}_{K_p}$ is flat over $\calO_{E_p}$ with special
fiber isomorphic to ${\Fl_v}_{\bar{\mu}_{r,s}}$.

The partially flag variety $\calP_{\mu_{r,s}}$ is a Schubert variety
in the affine Grassmanian of $G_{E_p}$ over $E_p$, (see
\eqref{affGrass}). Therefore, $\rH^*(\calP_{\mu_{r,s}})$ is a
natural representation of $\GL_n\times\bbG_m$ (the dual group of
$G_{E_p}$), which is indeed just $V_{r,s}$. Since nearby cycles
commute with proper push-forward, we have
\begin{equation}\label{isom}V_{r,s}\simeq \rH^*(\calP_{\mu_{r,s}})\simeq
\rH^*(\Psi^{\on{geom}}).\end{equation} Part (i) of the theorem would
follow if we can show that this isomorphism is an isomorphism of
$\on{GO}_n$ or $\on{GSp}_n$-modules. This is indeed the case, and
can be shown using the constructions in \cite{PZ}. In fact, as a
further application of the main results of this paper, we will prove
the corresponding Part (i) for all ramified groups in \emph{loc.
cit.}. Here for the ramified unitary groups, we give a more direct
(and easier) argument, without showing that \eqref{isom} is an
isomorphism of $\on{GO}_n$ or $\on{GSp}_n$-modules. However, we do
need the existence of this natural isomorphism.

We will first assume that $n$ is odd.  Recall that $V_{r,s}$ remains
irreducible as a representation of $\on{GO}_n$. Therefore,
$\calR\calS(V_{r,s})\simeq \IC_{\bar{\mu}_{r,s}}$ and we have
\[\dim V_{r,s}=\dim\rH^*(\IC_{\bar{\mu}_{r,s}})\leq \dim\rH^*(\Psi^{\on{geom}})=\dim \rH^*(\calP_{\mu_{r,s}})=\dim V_{r,s}.\]
Therefore, $\Psi^{\on{geom}}=\IC_{\bar{\mu}_{r,s}}$. As
$\IC_{\bar{\mu}_{r,s}}$ is irreducible, the inertial group $I$ acts
on $\Psi_{\on{M}^{\on{loc}}_{K_p}}$ via some character. Observe that
the action of $I$ on
$\rH^0(\Psi_{\on{M}^{\on{loc}}_{K_p}})\simeq\rH^0(\calP_{\mu_{r,s}})\simeq\bbQ_\ell$
is via the same character, again due to the fact that nearby cycles
commute with proper push-forward. Therefore the action of $I$ on
$\Psi_{\on{M}^{\on{loc}}_{K_p}}$ is trivial. Therefore,
$\Psi_{\on{M}^{\on{loc}}_{K_p}}[rs](\frac{rs}{2})\simeq
\IC_{\bar{\mu}_{r,s}}\otimes\calL$ for some rank one local system
$\calL$ on $\Spec \bbF_p$. By comparing the action of Frobenius on
$\rH^*(\calP_{\mu_{r,s}})$ and on $\rH^*(\IC_{\bar{\mu}_{r,s}})$, we
obtain the result in this case.

If $n$ is even, we need a modified argument since $V_{r,s}$ is not
irreducible as a $\on{GSp}_n$-module. In fact, we know that the
symplectic form induces a surjective map $V_{r,s}\to V_{r+2,s-2}$
and the kernel is the irreducible representation of $\on{GSp}_n$ of
highest weight $\bar{\mu}_{r,s}$, denoted by $W_{\bar{\mu}_{r,s}}$.
Recall that under the (ramified) geometric Satake isomorphism, the
cohomological grading corresponding to the grading by
$2\rho:\bbG_m\to\on{GSp}_n\subset\GL_n\times\bbG_m$. Therefore,
\eqref{isom} is an isomorphism of the representations of
$2\rho(\bbG_m)\subset\on{GSp}_n$. We claim that this already implies
Part (1) of the theorem. Indeed, for a representation $V$ of
$\on{GSp}_n$, we denote $V(i)$ to be the eigenspace of $2\rho$ of
eigenvalue $i$. Let us write
\[\Psi^{\on{geom}}=\sum_{m'\geq 0,m-m'\in2\bbZ_{\leq 0}} c_{m'}\IC_{\bar{\mu}_{n-m',m'}},\]
we need to show that $c_{m'}=1$. First as in Lemma \ref{mult},
$c_s=1$ by Proposition \ref{geom} (iii). Next, we show that
$c_{s-2}=1$. Observe that the gradings on
$\rH^*(\IC_{\bar{\mu}_{r,s}})$ range from $rs$ to $-rs$. From
\[\dim\rH^{(s-2)(r+2)}(\Psi^{\on{geom}})=c_{s-2}+\dim\rH^{(s-2)(r+2)}(\IC_{\bar{\mu}_{r,s}}),\]
\[\dim V_{r,s}((s-2)(r+2))=1+\dim (W_{\bar{\mu}_{r,s}}((s-2)(r+2))),\]
we conclude that $c_{s-2}=1$. Now by induction, $c_{m'}=1$ for all
$m'\geq 0, m-m'\in 2\bbZ_{\geq 0}$. This shows that
\[\Psi^{\on{geom}}\simeq \calR\calS(V_{r,s}).\]
Next, we determine the action of the inertial group $I$ on
$\Psi_{\on{M}^{\on{loc}}_{K_p}}$. First, assume that $r\neq s$. To
show that the action of $I$ on $\Psi_{\on{M}^{\on{loc}}_{K_p}}$ is
trivial, we again observe that the action of $I$ on each irreducible
direct summand of $\Psi_{\on{M}^{\on{loc}}_{K_p}}$ via certain
characters. On the other hand, the group $G_{E_p}$ is split.
Therefore, the action of $I$ on $\rH^*(\calP_{\mu_{r,s}})$ is
trivial, and therefore is trivial on
$\Psi_{\on{M}^{\on{loc}}_{K_p}}$. Again, comparing the action of the
Frobenius, we conclude the result in this case.

Finally, let us assume that $r=s=m$, where $n=2m$. Then $E_p=\bbQ_p$
and $G_{\bbQ_p}$ is not split. In addition the action of $I$ on
$\rH^*(\calP_{\mu_{m,m}})$ is not trivial. Indeed, as
$\calP_{\mu_{m,m}}$ is defined over $\bbQ_p$, according to the
Appendix, $V_{\mu_{m,m}}=\rH^*(\calP_{\mu_{m,m}})$ is a natural
representation of
${^L}G^{geom}=(\GL_{n}\times\bbG_m)\rtimes_{\on{act}^{geom}}\Gal(\overline{\bbQ}_p/\bbQ_p)$,
so that the natural action of $\Gal(\overline{\bbQ}_p/\bbQ_p)$ on
$\rH^*(\calP_{\mu_{m,m}})$ is given by the restriction of this
representation to $\Gal(\overline{\bbQ}_p/\bbQ_p)$. This semidirect
product
$(\GL_{n}\times\bbG_m)\rtimes_{\on{act}^{geom}}\Gal(\overline{\bbQ}_p/\bbQ_p)$
is not the Langlands dual group ${^L}G_{\bbQ_p}^{alg}$ of
$G_{\bbQ_p}$. But if we form both semi-product using
$I\subset\Gal(\overline{\bbQ}_p/\bbQ_p)$, they become the same
because the cyclotomic character is trivial on $I$. On the other
hand, since $G$ is split over $F_p$, the action of $I$ factors
through $I\to \Gal(F_p/\bbQ_p)\simeq \bbZ/2$.

\begin{lem}\label{Gal} The representation $V_{m,m}$ of
$(\GL_{n}\times\bbG_m)\rtimes I$, when restricted to
$\on{GSp}_n\rtimes I=\on{GSp}_n\times I$, decomposes as
\begin{equation}\label{decom}V_{m,m}=\sum_{m'\geq 0,m-m'\in 2\bbZ_{\geq
0}}W_{\bar{\mu}_{n-m',m'}}\otimes \chi_{m'},\end{equation} where
$\chi_{m'}$ is the trivial character of $\Gal(F_p/\bbQ_p)$ if $4\mid
m-m'$, and is the non-trivial character if $4\nmid m-m'$.
\end{lem}
\begin{proof}Clearly, there are some characters $\chi_{m'}$ of $\Gal(F_p/\bbQ_p)$ such
that the decomposition \eqref{decom} holds. We need to identify
these characters.

First, it is clear that $\chi_m=1$. This is because the lowest
weight space of $V_{\bar{\mu}_{m,m}}$ is the same as the lowest
weight space of $V_{m,m}$, which in turn is the same as
$\rH^0(\calP_{\mu_{m,m}})$ as $I$-modules. But the action of $I$ on
$\rH^0(\calP_{\mu_{m,m}})$ is trivial. This shows that $\chi_m=1$.

Now pick up $g\in I$ whose projection to $\Gal(F_p/\bbQ_p)\simeq
\bbZ/2$ is non-trivial. To identify other $\chi_{m'}$, let us write
the weight lattice of $\GL_n\times\bbG_m$ in a standard way to be
$\xch=\bigoplus\bbZ\varepsilon_i\bigoplus\bbZ\varepsilon$ and the
set of simple roots to be $\{\varepsilon_i-\varepsilon_{i+1}, 1\leq
i\leq n-1\}$. Then the action of $g$ on $\xch$ will send
$\varepsilon_i$ to $-\varepsilon_{n+1-i}$ and $\varepsilon$ to
$\varepsilon+\varepsilon_1+\cdots+\varepsilon_n$. The weight lattice
of $\on{GSp}_n$ is $\xch/\{\varepsilon_i+\varepsilon_{n+1-i}=0,
i=1,\ldots,m\}$.

Let $\{v_1,\ldots,v_n\}$ be a standard basis of $V$ so that $v_i$ is
a weight vector of $\GL_n$ of weight $\varepsilon_i$ as usual. Then
a basis of $V_{m,m}$ is given by $\{v_{i_1}\wedge\cdots\wedge
v_{i_m}\mid 1\leq i_1<\cdots<i_m\leq n\}$. We divide this set of
basis into two subsets $A$ and $B$. A base vector
$v_{i_1}\wedge\cdots\wedge v_{i_m}$ belongs to the subset $A$ if
$\{i,n+1-i\}\nsubseteq\{i_1,\ldots,i_m\}$ for any $1\leq i\leq m$.
All remaining base vectors belong to $B$. It is clear that
$\on{Span}\{v\mid v\in A\}\subset V_{\bar{\mu}_{m,m}}$ and
therefore, the action of $g$ fixes each $v\in A$ since $\chi_m=1$.
On the other hand, it is easy to see from the description of the
action of $g$ on $\xch$, that for $v\in B$, $gv$ will be a multiple
of some $w\in B, w\neq v$. From this, we deduce that for any $t$ in
the maximal torus of $\on{GSp}_n$,
\begin{equation}\label{1}
\tr(gt,V_{m,m})=\varepsilon(t)\sum\varepsilon_{1}(t)^{\pm
1}\cdots\varepsilon_{m}(t)^{\pm 1}.
\end{equation}
On the other hand, according to \eqref{decom}, we have
\[\tr(gt,V_{m,m})=\sum_{m'\geq 0,m-m'\in 2\bbZ_{\geq 0}}
\chi_{m'}(g)\on{ch}(W_{\bar{\mu}_{n-m',m'}})(t).
\]
where $\on{ch}(W_{\bar{\mu}_{r,s}})$ denotes the character of
$W_{\bar{\mu}_{r,s}}$ as a $\GSp_n$-module, and $\chi_{m'}(g)=\pm 1$
according to whether $\chi_{m'}$ is trivial or not. Now it is easy
to see that the above two identities force $\chi_{m'}=1$ if $4\mid
m-m'$ and $\chi_{m'}\neq 1$ if $4\nmid m-m'$. Indeed, let $T$ be an
indeterminant and write
\[\varepsilon(1+\varepsilon_1T)\cdots(1+\varepsilon_mT)(1+\varepsilon_1^{-1}T)\cdots(1+\varepsilon_m^{-1}T)=\sum a_kT^k,\]
then $a_{m+i}=a_{m-i}$ and
$\on{ch}(W_{\bar{\mu}_{r,s}})=a_s-a_{s-2}=a_r-a_{r+2}$. Put
$T=\sqrt{-1}$, the left hand side becomes
$(\sqrt{-1})^m\varepsilon(\varepsilon_1+\varepsilon_1^{-1})\cdots(\varepsilon_m+\varepsilon_m^{-1})$,
which is exactly \eqref{1}, and the right hand side is
$(\sqrt{-1})^m(a_m-2a_{m-2}+2a_{m-4}-\cdots)$. The lemma is proved.
\end{proof}

Finally, let us finish the prove the the theorem. We know that
\[\Psi^{\on{geom}}=\sum_{m'\geq 0,m-m'\in 2\bbZ_{\geq 0}}\IC_{\bar{\mu}_{n-m',m'}}.\]
As argued in the case $r\neq s$, the action of $I$ on
$\Psi_{\on{M}^{\on{loc}}_{K_p}}$ also factors through
$I\to\Gal(F_p/\bbQ_p)$. Assume that the action of $I$ on
$\IC_{\bar{\mu}_{n-m',m'}}$ is through the character $\chi'_{m'}$.
We need to show that $\chi'_{m'}=\chi_{m'}$. Since
$\rH^*(\Psi_{\on{M}^{\on{loc}}_{K_p}})\simeq
\rH^*(\calP_{\mu_{m,m}})\simeq V_{m,m}$ as $(2\rho(\bbG_m)\times
I)$-modules, by taking the $I$-invariants, we obtain that
\[\sum_{\chi'_{m'}=1}\rH^*(\IC_{\bar{\mu}_{n-m',m'}})=\sum_{\chi_{m'}=1} W_{\bar{\mu}_{n-m',m'}}.\]
Again, as argued before by considering the gradings, it is easy to
see that this forces $\chi'_{m'}=\chi_{m'}$. Finally, by comparing
the action of Frobenius on $\rH^*(\Psi_{\on{M}^{\on{loc}}_{K_p}}^I)$
and on $\rH^*(\calP_{\mu_{m,m}})^I$, we conclude the theorem.
\end{proof}

Combining Theorem \ref{Shi} and Theorem \ref{stalk cohomology}, it
is not hard to obtain the explicit formula of the trace of Frobenius
of $\Psi_{\on{M}^{\on{loc}}_{K_p}}$, which will be the input of the
Langlands-Kottwitz method of calculating the local Zeta function of
the Shimura varieties. Instead of write down the explicit formula,
let us characterize this function in terms of its trace on
``unramified" representations of $G(F)$ (which clearly determines
this function uniquely). The characterization verifies a conjecture
of Haines and Kottwitz in this case.

\begin{prop}\label{Haines-Kottwitz}
Let $z_{r,s}$ be the function on $G'(F)$ associated to
$\Psi_{\on{M}^{\on{loc}}_{K_p}}^I$ under the Grothendieck
sheaf-function dictionary, and let $V_{r,s}$ be the representation
of ${^L}G_{E_p}^{alg}$ attached to $\mu_{r,s}\in\xcoch(T)$ as above
(or in Corollary \ref{extension}). For $\pi$ an ``unramified"
representation of $G'(F)$, with the Langlands parameter $\Sat(\pi)$
as defined in \eqref{LP}, we have
\[\tr(\pi(z_{r,s}))=\tr(\Sat(\pi)(\Phi), V_{r,s}^I).\]
\end{prop}
The proof is a direct consequence of Theorem \ref{Shi} and
\eqref{trace}.

\begin{rmk}{\rm In the Langlands-Kottwitz methods of calculating the Zeta
factors of Shimura varieties, one needs some mysterious test
functions $z_\mu$ to be put into the trace formula. Assuming the
Local Langlands, Haines and Kottwitz give a conjectural
characterization of this test function $z_\mu$ in the general
setting (i.e. arbitrary group and arbitrary level structure). In the
special case when the group is quasi-split and the level structure
is special parahoric, in which case the Langlands parameters is
clear (as in Sect. \ref{parameter}), their characterization is
reduced to the above proposition. Therefore, this proposition is the
first example of their conjecture in the case when the group is
ramified at $p$. In \cite{PZ}, we will show that the same
characterization holds for arbitrary (tamely ramified) quasi-split
groups with special parahoric level structure.
 }\end{rmk}

Finally, let us make Theorem \ref{Shi} more explicit for some
special cases.

\begin{cor}Let $(r,s)=(n-1,1)$. Then the inertial group acts on
$\Psi_{Sh_{K_p}}$ trivially. In addition, As Weil sheaves,
$\Psi_{Sh_{K_p}}\simeq\bbQ_\ell$. In particular, for every $x\in
Sh_{K_p}(\bbF_{p^n})$, $\tr(\on{Frob}_x,\Psi_{K_p})=1$.
\end{cor}
\begin{proof}The first statement follows from Theorem \ref{Shi} and the local model diagram. We
need to show that
$\IC_{\bar{\mu}_{n-1,1}}\simeq\overline{\bbQ}_\ell[n-1](\frac{n-1}{2})$.
However, according to Theorem \ref{stalk cohomology}, we know that
$\IC_{\bar{\mu}_{n-1,1}}[1-n](\frac{1-n}{2})$ is a sheaf (for the
standard $t$-structure) rather than a complex, with each stalk
isomorphic to $\overline{\bbQ}_\ell$. As
$\IC_{\bar{\mu}_{n-1,1}}[1-n](\frac{1-n}{2})$ is indecomposable as
object in $D(\Fl_v)$, this forces
$\IC_{\bar{\mu}_{n-1,1}}[1-n](\frac{1-n}{2})\simeq\overline{\bbQ}_\ell$.
Now, since
$\pi^*\Psi_{K_p}\simeq\varphi^*\IC_{\bar{\mu}_{n-1,1}}[1-n](\frac{1-n}{2})$
and $\pi$ has geometrically connected fibers, the corollary follows.
\end{proof}

\begin{rmk}{\rm Concerning the part of the Frobenius trace, this corollary has been proven in \cite{Kr,PR2,Ri}.
Indeed, for the case $n$ is odd, and the special parahoric is
$\underline{G}_{v_0}$, this is a main result of \cite{Kr}. In this
case, $Sh_{K_p}$ is not semi-stable. For the case $n$ is odd and the
parahoric is $\underline{G}_{v_1}$, it is shown in \cite{Ri} that
$Sh_{K_p}$ is smooth. For the case $n$ is even, it is shown in
\cite{PR2} that $Sh_{K_p}$ is smooth.
 }\end{rmk}

Next, we consider the case $(r,s)=(2,2)$. Recall that the local
model diagram can be written as a morphism
\[Sh_{K_p}\to [\underline{G}_v\setminus\on{M}^{\on{loc}}_{K_p}],\]
where $[\underline{G}_v\setminus\on{M}^{\on{loc}}_{K_p}]$ denotes
the stack quotient. Therefore, the Schubert stratification on
$\on{M}^{\on{loc}}_{K_p}\otimes\overline{\bbF}_p$ induces a
stratification on $Sh_{K_p}\otimes\overline{\bbF}_p$, called the
Kottwitz-Rapoport (KR) stratification. In the case $(r,s)=(2,2)$,
the stratification has two strata $Sh_{K_p,b}$ and $Sh_{K_p,s}$. The
smaller one $Sh_{K_p,s}$ is zero-dimensional. Similarly to the
previous case, we have

\begin{cor}Let $(r,s)=(2,2)$. Then
$\Psi_{Sh_{K_p}}=\Psi_{Sh_{K_p}}^1+\Psi_{Sh_{K_p}}^2$. The inertial
action on $\Psi_{Sh_{K_p}}^1$ is trivial and as Weil sheaves,
$\Psi_{Sh_{K_p}}^1\simeq\bbQ_\ell$. The vanishing cycle
$\Phi_{Sh_{K_p}}(\bbQ_\ell)\simeq\Psi_{K_p}^2$. When we forget the
action of $\Gal(\overline{\bbQ}_p/\bbQ_p)$,
$\Psi_{Sh_{K_p}}^2=\sum_{x\in Sh_{K_p,s}}\delta_x[-4]$, where
$\delta_x$ is the delta sheaf supported at $x$. In addition, the
inertial action on $\delta_x$ factors through a non-trivial
quadratic character. In particular, for every $x\in
Sh_{K_p}(\bbF_{p^n})$, $\tr(\on{Frob}_x,\Psi_{Sh_{K_p}}^I)=1$.
\end{cor}

\appendix

\pagestyle{empty}
\section{Construction of the full Langlands dual group via the geometric Satake
correspondence\\
\medskip
By Timo Richarz, Xinwen Zhu}

\medskip

In the main body of the paper, we considered a reductive group $G$
over $F=k\ppart$ ($k$ algebraically closed), split over a tamely
ramified extension, and recovered $(H^\vee)^{\Gal(F^{s}/F)}$ by the
Tannakian formalism from a certain category of perverse sheaves
associated to $G$, where $H^\vee$ is the dual group of $G$, on which
$\Gal(F^{s}/F)$ acts via pinned automorphisms. In this appendix, we
take a different point of view to recover the full Langlands dual
group ${^LG}=H^\vee\rtimes\Gal(F^{s}/F)$ of $G$ by the Tannakian
formalism. The construction is easy but we can not find it in the
literature. Most proofs will be omitted or rather sketched since
they are very simple.

\medskip

Let us begin with a review of certain general nonsense of Tannakian
formalism. A similar discussion appears in \cite[Appendix 2]{HNY}.
Let $(\calC,\omega)$ be a neutralized Tannakian category over a
field $E$ of characteristic zero with fiber functor $\omega$. We
define a monoidal category $\Aut^\otimes(\calC,\omega)$ as follows:
objects are pairs $(\sigma,\al)$, where $\sigma:\calC\to\calC$ is a
tensor automorphism and $\al:\omega\circ\sigma\simeq\omega$ is a
natural isomorphism of tensor functors; morphisms between
$(\sigma,\al)$ and $(\sigma',\al')$ are natural tensor isomorphisms
between $\sigma$ and $\sigma'$ that are compatible with $\al,\al'$
in an obvious way. The monoidal structure is given by compositions.
Since $\omega$ is faithful, $\Aut^\otimes(\calC,\omega)$ is
(equivalent to) a set, and in fact is a group. For example,
$(\sigma,\al)=\id$ means that there is an isomorphism
$\varepsilon:\sigma\simeq\id$ of tensor functors such that
$\omega\varepsilon=\al$ (such $\varepsilon$ will be unique).

Let $H=\Aut^\otimes_{\calC}\omega$, the Tannakian group defined by
$(\calC,\omega)$. Let $\Aut(H)$ be the group of automorphisms of $H$
and $\on{Out}(H)$ be the group of outer automorphisms of $H$.

\begin{lem}\label{action}
There is a canonical action of $\Aut^\otimes(\calC,\omega)$ on $H$
by automorphisms. In addition, the map
$\Aut^\otimes(\calC,\omega)\to\Aut(H)$ induces
$[\Aut^\otimes(\calC)]\to \on{Out}(H)$, where
$[\Aut^\otimes(\calC)]$ is the group of isomorphism classes of
tensor automorphisms of $\calC$.
\end{lem}
The action of $(\sigma,\al)$ on $H$ is given as follows. Let $R$ be
a $E$-algebra and $h:\omega_R\simeq\omega_R$ be an $R$-point of $H$.
Then $(\sigma,\al)h$ is the following composition
\[\omega_R\stackrel{\al}{\leftarrow}\omega_R\circ\sigma\stackrel{h\circ\id}{\to}\omega_R\circ\sigma\stackrel{\al}{\rightarrow}\omega_R.\]

\begin{rmk}{\rm As is shown \cite{HNY}, $\Aut^\otimes(\calC,\omega)$ can
be upgraded into a \emph{fppf} sheaf on the category of affine
schemes over $E$, and as \emph{fppf} sheaves
$\Aut^\otimes(\calC,\omega)\simeq \Aut(H)$. We do not need this
fact.
 }\end{rmk}

Let $\Ga$ be an abstract group. We define an action of $\Ga$ on
$(\calC,\omega)$ to be a group homomorphism
$\on{act}:\Ga\to\Aut^\otimes(\calC,\omega)$. Assume that $\Ga$ acts
on $(\calC,\omega)$. We can then define $\calC^\Ga$, the category of
$\Ga$-equivariant objects in $\calC$ as follows: objects are
$(X,\{c_\ga\}_{\ga\in\Ga})$, where $X$ is an object in $\calC$ and
$c_\ga:\on{act}_\ga(X)\simeq X$ is an isomorphism, satisfying the
natural cocycle condition, i.e.
$c_{\ga'}\circ\on{act}_{\ga'}(c_\ga)=c_{\ga'\ga}$; the morphisms
between $(X,\{c_\ga\}_{\ga\in\Ga})$ and
$(X',\{c'_\ga\}_{\ga\in\Ga})$ are morphisms between $X$ and $X'$,
compatible with $c_\ga,c'_\ga$ in an obvious way.

\begin{lem}\label{Tannaka}
Let $\Ga$ be a group acting on $(\calC,\omega)$. Then the category
$\calC^\Ga$ is a neutral Tannakian category, with fiber functor
$\omega$. In addition, if $\Ga$ is a finite group, regarded as an
algebraic $E$-group, then the Tannakian group
$\Aut_{\calC^\Ga}^\otimes\omega$ is canonically isomorphic to
$H\rtimes\Ga$.
\end{lem}
\begin{proof}
The monoidal structure on $\calC^\Ga$ is defined as
\[(X,\{c_\ga\}_{\ga\in\Ga})\otimes (X',\{c'_\ga\}_{\ga\in\Ga})=(X'',\{c''_\ga\}_{\ga\in\Ga}),\]
where $X''=X\otimes X'$ and $c''_\ga:\on{act}_\ga(X'')\to X''$ is
the composition
\[\on{act}_\ga(X\otimes X')\simeq
\on{act}_\ga(X)\otimes\on{act}_\ga(X')\stackrel{c_\ga\otimes
c'_\ga}{\longrightarrow} X\otimes X'.\] This gives $\calC^\Ga$ the
structure of a Tannakian category. Now assume that $\Ga$ is finite,
and hence $H\rtimes\Ga$ is an affine group scheme. By \cite[Prop.
2.8]{DM}, it is enough to show that,
\[\Rep(H)^\Ga\stackrel{\simeq}{\longrightarrow}\Rep(H\rtimes\Ga)\]
as tensor categories compatible with the forgetful functors. Let
$((V,\rho),\{c_\ga\}_{\ga\in\Ga})\in\Rep(H)^\Ga$. Then we define
$(V,\tilde{\rho})\in\Rep(H\rtimes\Ga)$, for any $k$-algebra $R$
$(h,\ga)\in (H\rtimes\Ga)(R)$, by
\[(h,\ga)\longmapsto \rho(h)\circ\alpha_{g,R}(V)\circ\omega_R\circ c_\ga^{-1}\in\GL(V\otimes R),\]
where $\alpha_{g,R}:\omega_R\circ\sigma_g\simeq\omega_R$ is induced
by the action of $\Ga$ as above. Using the cocycle relation one
checks that this is indeed a representation, and that the map
defines the desired equivalence.
\end{proof}

\begin{rmk}{\rm  If $\Ga$ is not finite, then the category
$(\calC^\Ga,\omega)$ is still Tannakian, but
$\tilde{H}=\Aut^\otimes_{\calC^\Ga}\omega$ is no longer
$H\rtimes\Ga$ since the latter cannot be regarded as an affine group
scheme (sometimes $\tilde{H}$ is called the algebraic envelop of
$H\rtimes \Ga$). However, there is always a group homomorphism
\[H(E)\rtimes\Ga\to \tilde{H}(E).\]
Although this is not an isomorphism in general, we may still regard
$\omega(X)$ for $X\in\calC^\Ga$ as a representation of
$H(E)\rtimes\Ga$.
 }\end{rmk}

\medskip

Now, we assume that $G$ is a connected reductive group over
\emph{any} field $k$. We switch the notation to use $G^\vee$ to
denote the reductive group over $E=\overline{\bbQ}_\ell$ dual to $G$
in the sense of Langlands, i.e. the root datum of $G^\vee$ is dual
to the root datum of $G_{k^{s}}$. Up to the choice of a pinning
$(G^\vee,B^\vee,T^\vee,X^\vee)$ of $G^\vee$, we have an action of
$\Ga_k=\Gal(k^{s}/k)$ on $G^\vee$ via
\begin{equation}\label{pinned}
\Ga_k\to\on{Out}(G_{k^{s}})\simeq
\on{Out}(G^\vee)\simeq\Aut(G^\vee,B^\vee,T^\vee,X^\vee)\subset
\Aut(G^\vee).
\end{equation}
Then the Langlands dual group ${^L}G$ is defined to be
$G^\vee\rtimes\Ga_k$. Our goal is to recover this group via the
above Tannakian formalism.

Let $L^+G$ be the jet group of $G\otimes k[[t]]$ and $LG$ be the
loop group of $G\otimes k\ppart$. Recall that by definition, for
every $k$-algebra $R$, $L^+G(R)=G(R[[t]])$ and $LG(R)=G(R\ppart)$.
Let
\begin{equation}\label{affGrass}
\Gr_G=LG/L^+G
\end{equation}
be the affine Grassmannian of $G$ over $k$. Let $\Gr_G\otimes k^{s}$
be its base change to the separable closure of $k$. From \cite[Lemma
3.3]{Z}, formation of the affine Grassmannian commutes with
\'{e}tale base change, we have $\Gr_{G}\otimes k^{s}\simeq
\Gr_{G_{k^{s}}}$. Since $G_{k^{s}}$ is split, we can consider the
usual Satake category $\Sat$ on $\Gr_{G_{k^{s}}}$, i.e. the category
of $(L^+G\otimes k^s)$-equivariant perverse sheaves on $\Gr_G\otimes
k^s$, which is equivalent to $\Rep(G^\vee_{\overline{\bbQ}_\ell})$
via the geometric Satake correspondence of \cite{MV}. Note that the
Galois group $\Ga_k$ acts on $\Gr_{G}\otimes k^{s}$. For
$\ga\in\Ga_k$, the pullback functor $\ga^*:D(\Gr_G\otimes k^{s})\to
D(\Gr_G\otimes k^{s})$ clearly restricts to a functor
$\ga^*:\Sat\to\Sat$. In addition, there is a canonical isomorphism
$\al_\ga: \rH^*(\ga^*\calF)\simeq \rH^*(\calF)$.

\begin{lem}\label{action of Galois}
The assignment $\ga\mapsto (\ga^*,\al_\ga)$ defines an action of
$\Ga_k$ on $(\Sat, \rH^*)$.
\end{lem}

According to Lemma \ref{action}, there is a canonical action of
$\Ga_k$ on $G^\vee$, denoted by $\on{act}^{geom}$. And we can form
\[{^L}G^{geom}:=G^\vee\rtimes_{\on{act}^{geom}}\Ga_k,\]
which we will call the geometric Langlands dual group.

Now, our goal is to understand the relation between ${^L}G^{geom}$
and the usual Langlands dual group ${^L}G$. Recall that in Sect.
\ref{proof}, we explained that once we choose an ample line bundle
on $\Gr_G$, the geometric Satake isomorphism provides $G^\vee$ with
a canonical pinning $(G^\vee,B^\vee,T^\vee,X^\vee)$. Therefore,
there is an action of $\Ga_k$ on $G^\vee$ via \eqref{pinned},
denoted by $\on{act}^{alg}$. Then we can form the usual Langlands
dual group by ${^L}G^{alg}=G^\vee\rtimes_{\on{act}^{alg}}\Ga_k$. It
turns out that the difference between $\on{act}^{geom}$ and
$\on{act}^{alg}$ can be described explicitly.

Let
\[\on{cycl}:\Ga_k\to\bbZ_\ell^\times\]
be the cyclotomic character of $\Ga_k$ defined by the action of
$\Ga_k$ on the $\ell^\infty$-roots of unity of $k^{s}$. Let
$G^\vee_{\ad}$ be the adjoint group of $G^\vee$. Let $\rho$ be the
half sum of positive coroots of $G^\vee$, which gives rise to a
one-parameter group $\rho:\bbG_m\to G^\vee_\ad$. We define a map
\[\chi:\Ga_k\stackrel{\on{cycl}}{\to} \bbZ_\ell^\times\stackrel{\rho}{\to} G^\vee_{\ad}(\overline{\bbQ}_\ell),\]
which gives a map $\Ad_{\chi}:\Ga_k\to\Aut(G^\vee)$ to the inner
automorphism of $G^\vee$.
\begin{prop}\label{comparison}Let $(G^\vee, B^\vee, T^\vee, X^\vee)$
be the pinning defined by the geometric Satake isomorphism as in \S
\ref{proof}. We have $\on{act}^{geom}=\on{act}^{alg}\circ
\Ad_{\chi}$.
\end{prop}
\begin{proof}Observe that the action of
$\Ga_k$ preserves the cohomological grading. In addition, $\Ga_k$
acts on $X^\vee$ through $\on{cycl}$ since $X^\vee$ is the Chern
class of the chosen ample line bundle on $\Gr_G$. Therefore,
$\Ad_{\chi^{-1}}\circ \on{act}^{geom}$ preserves $(G^\vee, B^\vee,
T^\vee, X^\vee)$. But since $\Ad_{\chi^{-1}}\circ \on{act}^{geom}$
and $\on{act}^{alg}$ act on the the based root datum $(\xch(T),
\Delta)$ by the same way, these two actions must coincide.
\end{proof}

\begin{rmk}{\rm (i) An interesting corollary of the above proposition is
that the $\on{act}^{geom}$ of $\Ga_k$ on $G^\vee$ only depends on
the quasi-split form of $G$, since the same is true for
$\on{act}^{alg}$.

(ii) The existence of these two actions $\on{act}^{geom}$ and
$\on{act}^{alg}$ is a geometric source of the two natural
normalizations for the Satake parameters. In addition, these two
actions are also parallel to the notions of $C$-algebraic and
$L$-algebraic, as recently introduced by Buzzard and Gee \cite{BG}.
 }\end{rmk}

\begin{cor}\label{geom-alg}
We have $
 {^L}G^{alg}\stackrel{\simeq}{\longrightarrow} {^L}G^{geom}$ given
 by
\[(g,\ga)\mapsto (\Ad_{\chi(\ga^{-1})}(g),\ga).\]\end{cor}

Following \cite[2.6]{Bo}, we define the notation of the algebraic
representations of ${^L}G^{alg}$ as follows. For every $k'\subset
k^s$ finite over $k$ such that $G_{k'}$ is split, one can form the
``finite form" of Langlands dual group
$G^\vee\rtimes_{\on{act}^{alg}}\Gal(k'/k)$, which can be regarded as
an algebraic group over $\overline{\bbQ}_\ell$. Therefore, we may
consider
${^L}G^{alg}=\underleftarrow{\lim}G^\vee\rtimes_{\on{act}^{alg}}\Gal(k'/k)$
as a pro-algebraic group over $\overline{\bbQ}_\ell$, Then it makes
sense to talk about the category of algebraic representations
$\Rep({^L}G^{alg})$ of ${^L}G^{alg}$, which is the inductive limit
of $\Rep(G^\vee\rtimes_{\on{act}^{alg}}\Gal(k'/k))$.

Now, we consider certain categories of perverse sheaves. First,
according to the above discussions, we call the action of $\Ga_k$ on
$\Sat$ via $\ga\mapsto (\ga^*,\al_\ga)$ as in Lemma \ref{action of
Galois} the geometric action. We can also define an algebraic action
of $\Ga_k$ on $\Sat$ as $\ga\mapsto
(\ga^*,\chi(\ga)^{-1}\al_\ga\chi(\ga))$. Clearly, there is a
canonical isomorphism between $\Sat^{\Ga_k,geom}$ and
$\Sat^{\Ga_k,alg}$ sending $(\calF,\{c_\ga\}_{\ga\in\Ga_k})$ to
$(\calF,\{c_\ga\}_{\ga\in\Ga_k})$. The reason we distinguish them is
due to the following observation: since the algebraic action of
$\Ga_k$ on $G^\vee$ factors through $\Gal(k'/k)$ if $G_{k'}$ is
split, the algebraic action of $\Ga_k$ on $\Sat$ also factors
through $\Gal(k'/k)$, and therefore it makes sense to talk about
$\Sat^{\Gal(k'/k),alg}$, which is naturally a full category of
$\Sat^{\Ga_k,alg}$. In addition, according to Lemma \ref{Tannaka},
$\Sat^{\Gal(k'/k),alg}$ is equivalent to
$\Rep(G^\vee\rtimes_{\on{act}^{alg}}\Gal(k'/k))$. Now, let
$\Sat^{\Ga_k,alg,f}\subset\Sat^{\Ga_k,alg}$ be the full subcategory,
which is the union of all $\Sat^{\Gal(k'/k),alg}$. We obtain that
\[\rH^*:\Sat^{\Ga_k,alg,f}\simeq\Rep({^L}G^{alg}).\]
We next goal then is to identify $\Sat^{\Ga_k,alg,f}$ as a
subcategory of $\Sat^{\Ga_k,geom}$.

As $\Ga_k$ is a topological group, a natural guess would be the full
subcategory of $\Sat^{\Ga_k,geom,ct}$, consisting of objects on
which $\Ga_k$ acts continuously. Equivalent, let
$\calP_{L^+G}(\Gr_G)$ be the category of perverse sheaves on
$\Gr_G$.  Then the pullback functor to $\Gr_G\otimes k^s$ induces
$\calP_{L^+G}(\Gr_G)\simeq \Sat^{\Ga_k,geom,ct}$. However, it is not
the case that $\Sat^{\Ga_k,geom,ct}=\Sat^{\Ga_k,alg,f}$.

\begin{dfn}(i) Let $X$ be a smooth variety over $k$, we define a
constant sheaf to be a direct sum of
$(\overline{\bbQ}_\ell[1](\frac{1}{2}))^{\otimes \dim X}$.

(ii) We define $\calP^f_{L^+G}(\Gr_G)$ to be the full subcategory of
$\calP_{L^+G}(\Gr_G)$, consisting of those sheaves $\calF$, such
that there exists some $k'\supset k$ such that $\calF\otimes k'$ is
constant along each $(L^+G\otimes k')$-orbit.
\end{dfn}
\begin{rmk}{\rm
(i) The toy model is when $G=\{e\}$ is the trivial group. Then
$\calP_{L^+G}(\Gr_G)$ is the category $\Ga_k\Mod$ of continuous
representations of $\Ga_k$ while $\calP^f_{L^+G}(\Gr_G)$ is the
subcategory $\Ga_k\Mod^f$ consisting of representations of finite
quotients of $\Ga_k$. In particular, if $k$ is a finite field, we
can identify the latter as the category of semi-simple
$\Ga_k$-modules, pure of weight zero.

(ii) Observe that every object in $\calP_{L^+G}(\Gr_G)$ is of the
form $\oplus_i \IC_i\otimes \calL_i$, where $\IC_i$ is an
intersection cohomology sheaf on $\Gr_G$, and $\calL_i$ is a
representation of $\Ga_k$. Therefore,
$\calF\in\calP_{L^+G}^f(\Gr_G)$ if and only if all
$\calL_i\in\Ga_k\Mod^f$. In particular, if $k=\bbF_q$ is a finite
field, the category $\calP^f_{L^+G}(\Gr_G)$ is equivalent to the
category of semi-simple $L^+G$-equivariant perverse sheaves on
$\Gr_G$, pure of weight zero.
 }\end{rmk}

\begin{prop}Under the canonical isomorphism
$\Sat^{\Ga_k,geom}=\Sat^{\Ga_k,alg}$, we have the identification
\[\calP_{L^+G}^f(\Gr_G)=\Sat^{\Ga_k,alg,f}.\]
Therefore, we obtain an equivalence of tensor categories
\[\rH^*:\calP_{L^+G}^f(\Gr_G)\to \Rep({^L}G^{alg}).\]
\end{prop}
\begin{proof}Let $k'\supset k$ such that $G_{k'}$ splits. We denote $\calP_{L^+G,k'}(\Gr_G)$ the full subcategory
of $\calP^f_{L^+G}(\Gr_G)$ consisting of those $\calF$ such that
$\calF\otimes k'$ is constant along each $(L^+G\otimes k')$-orbit.
Then under
\[\calP_{L^+G}(\Gr_G)\to \Sat^{\Ga_k,geom}\simeq\Sat^{\Ga_k,alg},\]
$\calP_{L^+G,k'}(\Gr_G)$ maps to $\Sat^{\Gal(k'/k),alg}$. To see
this, one reduces to the case when $G$ is split over $k$ and $k'=k$.
In this case $\calP_{L^+G,k}(\Gr_G)\simeq\Sat$ and this statement is
clear.
\end{proof}

For every $\calF\in\calP_{L^+G}^f(\Gr_G)$, let us describe
$\rH^*(\calF)$ as a representation of ${^L}G^{alg}$ more explicitly.
First, as an object in $\Sat^{\Ga_k,geom}$, it is a natural
representation of ${^L}G^{geom}$, on which $G^\vee$ acts via the
usual geometric Satake isomorphism, and $\Ga_k$ acts via the natural
Galois action. Then the action of ${^L}G^{alg}$ is via
\eqref{geom-alg}.

In particular, if $\bGr_\mu$ is a Schubert variety in $\Gr_G$
defined over $k$ (i.e., the conjugacy class of the one-parameter
subgroup determined by $\mu$ is defined over $k$), then
$\IC_{\bGr_\mu}$ is an object in $\calP_{L^+G}^f(\Gr_G)$. Therefore,
$\rH^*(\IC_{\bGr_{\mu}})$ is a representation of
$G^\vee\rtimes_{\on{act}^{alg}}\Gal(k'/k)$, where $k'$ is the
splitting field of $G$. We thus obtain
\begin{cor}\label{extension}
Let $V_\mu$ be a representation of $G^\vee$ of highest weight $\mu$.
If the conjugacy class of the one parameter subgroup $\mu:\bbG_m\to
G_{k^s}$ is defined over $k$, then $V_\mu$ can be extended
canonically to a representation of
$G^\vee\rtimes_{\on{act}^{alg}}\Gal(k'/k)$, where $k'$ is the
splitting field of $G$.
\end{cor}

Now we specialize to the case that $k=\bbF_q$ is a finite field, so
that $G\otimes k[[t]]$ is a hyperspecial group scheme for the
unramified group $G_{k\ppart}$. Note that
${^L}G_{k\ppart}^{alg}={^L}G^{alg}$. We therefore obtain the
geometric Satake isomorphism for unramified groups.

\begin{thm}\label{unramified Satake}
Let $\calP^f_{L^+G}(Gr_G)$ be the category of semi-simple,
$L^+G$-equivariant perverse sheaves on $\Gr_G$, pure of weight zero.
Then we have an equivalence of tensor categories
\[\rH^*:\calP^f_{L^+G}(\Gr_G)\simeq \Rep({^L}G^{alg}).\]
\end{thm}

Let $\sigma$ be the Frobenius element in $\Ga_k$. Denote by
$\calH_G$ the Grothendieck ring of $\calP^f_{L^+G}(\Gr_G)$, tensored
with $\overline{\bbQ}_\ell$, and by $\calR_{^LG}$ the algebra
associated to $\Rep({^L}G^{alg})$. Theorem \ref{unramified Satake}
gives an isomorphism of algebras
\begin{equation}\label{uncat}
\Phi:\calH_G\to\calR_{{^L}G}.
\end{equation} Let $\rH_G$ be the
spherical Hecke algebra of compactly supported
bi-$G(\bbF_q[[t]])$-invariant functions. Let $\on{R}_{^LG}$ be the
algebra of $\overline{\bbQ}_\ell$-valued functions on $(G^\vee\times
\sigma)_{ss}$, generated by the characters of elements in
$\Rep({^L}G^{alg})$. Here, $(G^\vee\times \sigma)_{ss}$ is the set
of semi-simple elements in $G^\vee\times\sigma$, as defined in
\cite[Sect. 6]{Bo}. We have a surjective map of algebras $\Tr:
\calH_G\to\rH_G$ (resp. $\on{Ch}:\calR_{^LG}\to\on{R}_{^LG}$) given
by the trace of Frobenius (resp. by sending a representation to its
character).
\begin{lem}\label{classicalsatake}
The isomorphism \eqref{uncat} induces a unique isomorphism
\[\phi:\rH_G\to\on{R}_{^LG}.\]
such that $\on{Ch}\circ\Phi=\phi\circ\Tr$.
\end{lem}
\begin{proof}Uniqueness is clear and we show the existence. For an object $X$ in either category, we denote by $[X]$ its class in the Grothendieck ring.
Then it is easy to see that the kernel of the map $\Tr:
\calR_{^LG}\to \on{R}_{^LG}$ is the ideal generated by elements of
the form $[V\otimes \psi]-\psi(\sigma)[V]$, where $V\in
\Rep({^L}G^{alg})$ and $\psi: \Ga_k\to \overline{\bbQ}_\ell^\times$
is a character of $\Ga_k$ factoring through a finite quotient. On
the other hand, the kernel of the map $\on{Ch}:\calH_G\to \rH_G$ is
the ideal generated by elements of the form
$[\calF\otimes\calL]-\tr(\sigma,\calL)[\calF]$, where $\calF\in
\calP^f_{L^+G}(\Gr_G)$, and $\calL$ is a rank one local system on
$\Spec \bbF_q$, pure of weight zero. But it is clear that these two
ideals match under $\Phi$. The lemma follows.
\end{proof}

Recall that by \cite[6.7]{Bo} the classical Satake isomorphism also
gives an isomorphism of algebras $\tilde{\phi}:\rH_G\simeq
\on{R}_{^LG}$. By tracking back the construction of geometric Satake
correspondence and the classical Satake isomorphism, one can show
that $\phi=\tilde{\phi}$. Indeed, if $G=T$ is a toruse, this is
clear. For general $G$, one observes that the fiber functor
decomposes as a direct sum of weight functors \cite[\S 3]{MV}:
$\calP_{L^+G}^f(\Gr_G)\to \calP_{L^+T}^f(\Gr_T)$, and under the
sheaf-function dictionary this corresponds the constant term map
$\on{CT}:\rH_G\to \rH_T$. Therefore, either $\phi$ or $\tilde{\phi}$
is uniquely determined by the following commutative diagram
\[\begin{CD}
\rH_G@>\sim >>\on{R}_{^LG}\\
@V\on{CT}VV@VV\res V\\
\rH_T@>\sim >>\on{R}_{^LT}
\end{CD}.\]
The general case follows.

\end{document}